\documentclass[twoside]{amsart}
\usepackage{amscd}
\usepackage[all]{xy}

\usepackage{graphicx}

\usepackage{xcolor}

\newcommand{\Hilb}{{\operatorname{Hilb}}}
\newcommand{\Gr}{{\operatorname{Gr}}}

\newcommand{\rat}{{\operatorname{rat}}}

\renewcommand{\Im}{{\operatorname{Im}}}
\newcommand{\Jac}{{\operatorname{Jac}}}
\newcommand{\Hom}{\operatorname{Hom}}
\newcommand{\End}{\operatorname{End}}
\newcommand{\CH}{\operatorname{CH}}

\newcommand{\Corr}{\operatorname{Corr}}
\newcommand{\Spec}{\operatorname{Spec}}
\newcommand{\Ker}{\operatorname{Ker}}

 \newcommand{\Alb}{\operatorname{Alb}}

\renewcommand{\dim}{\operatorname{dim}}
\newcommand{\Mot}{\operatorname{Mot}}
\renewcommand{\Alb}{\operatorname{Alb}}

 \newcommand{\Prym}{\operatorname{Prym}}
\renewcommand{\Im}{\operatorname{Im}}

\newcommand{\C}{\mathbf{C}}
\renewcommand{\L}{\mathbf{L}}
\renewcommand{\P}{\mathbf{P}}

\newcommand{\Q}{\mathbf{Q}}
\newcommand{\Z}{\mathbf{Z}}
\newcommand{\un}{\mathbf{1}}
\newcommand{\A}{\mathbf{A}}

 \newcommand{\sC}{\mathcal{C}}
\newcommand{\sA}{\mathcal{A}}
\newcommand{\sM}{\mathcal{M}}

\newcommand{\sO}{\mathcal{O}}

\newcommand{\sB}{\mathcal{B}}

 \newcommand{\sH}{\mathcal{H}}

 \newcommand{\sE}{\mathcal{E}}
\newcommand{\sI}{\mathcal{I}}

 \numberwithin{equation}{section}

\theoremstyle{plain}
\newtheorem{thm}[equation]{Theorem}
\newtheorem{prop}[equation]{Proposition}
\newtheorem{lm}[equation]{Lemma}
\newtheorem{cor}[equation]{Corollary}
\newtheorem{conj}[equation]{Conjecture}

\theoremstyle{definition}
\newtheorem{defn}[equation]{Definition}

\newtheorem{rk}[equation]{Remark}

\begin{document}

\title{The transcendental motive of a cubic fourfold}

\author[M. Bolognesi]{Michele Bolognesi}
\address{Institut Montpellierain Alexander Grothendieck \\ %
Universit\'e de Montpellier \\ %
Case Courrier 051 - Place Eug\`ene Bataillon \\ %
34095 Montpellier Cedex 5 \\ %
France}
\email{michele.bolognesi@umontpellier.fr}

\author[C. Pedrini]{Claudio Pedrini}
\address{Dipartimento di Matematica \\ %
Universit\'a degli Studi di Genova \\ %
Via Dodecaneso 35 \\ %
16146 Genova \\ %
Italy}
\email{pedrini@dima.unige.it}

\begin{abstract} 
In this note we introduce the transcendental part $t(X)$ of the motive of a cubic fourfold  $X$ and prove that it is isomorphic to the (twisted) transcendental part $h^{tr}_2(F(X))$ in a suitable Chow-K\"unneth decomposition for the motive of the Fano variety of lines $F(X)$. Then we prove that $t(X)$ is isomorphic to the {\it Prym motive} associated to the surface $S_l \subset F(X)$ of lines meeting a general line $l$. If $X$ is a special cubic fourfold in the sense of Hodge theory,  and $F(X) \simeq S^{[2]}$, with $S$ a K3, then we show that  $t(X)\simeq t_2(S)(1)$, where $t_2(S)$ is the transcendental motive. Therefore the motive $h(X)$ is finite dimensional if and only if $S$ has a finite dimensional motive. If $X$ is very general then $t(X)$ cannot be isomorphic to the (twisted) transcendental motive of a surface. We relate the existence of an isomorphism $t(X) \simeq t_2(S)(1)$ to conjectures by Hassett and Kuznetsov on the rationality of a special cubic fourfold. Finally we consider the case of cubic fourfolds $X$ admitting a fibration over $\P^2$, whose fibers are either quadrics or del Pezzo surfaces of degree 6, and prove the isomorphism $t_2(S)(1) \simeq t(X)$, with $S$ a K3 surface.

\end{abstract}

\maketitle 
\section {Introduction} 
We will work over the complex field. Cubic fourfolds are among the most mysterious objects in algebraic geometry. Despite the simplicity of the definition of such classically flavoured objects, the birational geometry of cubic fourfolds is extremely hard to understand and many modern techniques (Hodge theory, derived categories, etc. - see \it e.g. \rm \cite{Kuz,Has 2,AT} for details) have been successfully deployed in order to have a deeper understanding. In any case, the rationality of the generic cubic fourfold is still an open problem. Also the finite dimensionality of the motive $h(X)$  of a cubic fourfold, as conjectured by several authors (see \cite{Ki} , \cite{An}),  is known to hold only in some scattered cases.\par
We will denote by  $\sM_{rat}(\C)$   the  (covariant) category  of  Chow motives (with $\Q$-coefficients), whose objects are of the form $(X,p,n)$, where $X$ is a smooth projective variety over $\C$ of dimension $d$, $p$ is an idempotent in the ring $A^d(X \times X)= CH^d(X \times X)\otimes \Q $ and $n \in \Z$.  If $X$ and $Y$ are smooth  projective varieties  over $\C$, then the morphisms  $\Hom_{\sM_{rat}}(h(X),h(Y)) $ of their motives $h(X)$ and $h(Y)$ are given by correspondences in the Chow groups $A^*( X \times Y) =CH^*(X\times Y) \otimes \Q$. More precisely, in our covariant setting, we have
$$\Hom_{\sM_{rat}}(X,p,m),(Y,q,n)) =q\circ A_{d+m-n}(X\times 
Y)\circ p \subset A_{d+m-n}(X\times Y)$$
where  $d =\dim X$ and $\circ $ means composition of correspondences. The category $\sM_{rat}(\C)$ is additive,pseudo-abelian, rigid and has a tensor structure (see \cite{KMP}).
The unit motive is  $\un =(\Spec(\C , 1, 0)$: it is a unit for the tensor structure. The
{\it Lefschetz motive} $\L$ is   defined via the motive  of the projective line: $h (\P^1) = \un \oplus \L$ and there is  an isomorphism $\L\simeq(\Spec(k),1,1)$
 For every motive $M=(X,p,m)$ the Tate twist $M(r)$ is  the motive $(X,p,m+r)$. Note that,
with our  covariant convention, $M(r)\simeq M\otimes \L^{\otimes r}$ for $r\ge 0$.\par
\noindent The Chow groups of a motive $(X,p,m)\in \sM_{rat}(\C)$ are defined as follows
\begin{align*}
A^i(X,p,m) &= \Hom_{\sM_{rat}}((X,p,m),\L^i)= p^*A^{i-m}(X)\\
A_i(X,p,m)&= \Hom_{\sM_{rat}}((\L^i,(X,p,m))=p_*A_{i-m}(X).
\end{align*}
A similar definition holds for the category $\sM_{hom}(\C)$ of homological motives, with respect to singular cohomology $H^*(X)$, where
$$H^i(X,p,m)= p^*H^{i-2m}(X) \  ;   \  H_i(X,p,m)= p_*H_{i-2m}(X).$$ 
Let $X$ be a smooth projective variety over $\C$. We say that its motive $h(X)\in \sM_{rat}(\C)$ has a{ \it Chow-K\"unneth decomposition }(C-K for short)
if there exist orthogonal projectors
$\pi_i=\pi_i(X) \in \Corr_0(X,X)=A^d(X \times X)$, for $0 \le i \le 2d $,
such  that
$cl^d(\pi_i)$ is the $(i,2d-i)$-component of $\Delta_X$ in
$H^{2d}(X\times X)$ and
\[[\Delta _X] = \sum_{0 \le i \le 2d} \pi_i.\]
This implies that in $\sM_\rat$ the motive $h(X)$ 
decomposes as follows:
$$h(X) =\bigoplus_{0 \le i \le 2d} h_{i}(X),$$
where $h_{i}(X) =(X,\pi_i, 0)$. Moreover
\[H^*(h_i(X))=H^i(X), \qquad H_*(h_i(X))=H_i(X)\]
If we have $\pi_i=\pi_{2d-i}^t$ for all $i$, we say that the C-K
decomposition is \emph{self-dual}.\par
By the results in \cite[7.2.3]{KMP} every smooth projective surface $S$ has a {\it reduced C-K decomposition} $ h(S)=\sum_{0\le 4}h_i(S)$ with
$$h_2(S) =h^{alg}_2(S) \oplus t_2(S)= (S,\pi^{alg}_2) \oplus (S,\pi^{tr}_2).$$
Here  $h^{alg}_2(S) \simeq \L^{\rho(S)}$, where $\rho(S)$ is the rank of the Neron-Severi group and
$$H^2(S)=H^2_{alg}(S) \oplus H^2_{tr}(S) 
=\pi^{alg}_2 H^2(S)\oplus  \pi^{tr}_2 H^2(S)$$
The motive $t_2(S)$ is called the {\it transcendental motive} of $S$. It is a birational  invariant and
$$ H^*(t_2(S)) =H^2(t_2(S) =T(S)_{\Q} \  ; \  A^2(t_2(S)\simeq K(S),$$
where $T(S)$ is the transcendental lattice and  $K(S)$ is the Albanese kernel, i.e the kernel of the map $A_0(S)_{hom}  \to \Alb(S)$.\par
We recall the definition of finite dimensionality  introduced by S.Kimura in \cite{Ki}.
Let  $M \in \sM_{\rat}(\C)$ and let $\Sigma _n$ be the symmetric group of order $n$. Then  $\wedge^n M $
is the image of $M(X^n)$ under the projector
$$(1/ n!) (\sum_{\sigma \in \Sigma_n} sgn(\sigma) 
\Gamma_{\sigma}$$
while $S^nM$is its image under the projector
$$(1/ n!) (\sum_{\sigma \in \Sigma_n}  
\Gamma_{\sigma}.$$
A motive $M$ is said to be \emph{evenly (oddly) finite-dimensional} if
$\wedge^n M=0$ ($S^nX=0$) for some $n$.
 A motive $M$  is finite-dimensional if it can be decomposed into 
a direct sum
   $M_+\oplus M_-$ where $M_+$ is evenly finite-dimensional and
   $M_-$ is oddly finite-dimensional.
According to Kimura's conjecture in \cite{Ki} all motives should be finite dimensional. The conjecture is known to hold  for curves, rational surfaces, surfaces with $p_g(X)=0$, which are not of general type, abelian varieties and  some 3-folds. If  $ d= \dim X \le 3 $,  then the finite dimensionality of $h(X)$ is a birational invariant (see \cite[Lemma 7.1]{GG}), the reason being that in order  to make regular a birational map $X \to Y$  between smooth projective 3-folds one needs to blow up only  points and curves, whose motives are finite dimensional. \par
 If  $X$ is a complex Fano threefold, then  $h(X)$ is finite dimensional  and  of abelian type, i.e. it lies in  the subcategory of $\sM_{rat}(\C)$ generated by the motives of abelian varieties, see \cite[Thm. 5.1]{GG}. The proof is based on the fact that all the Chow groups $A_i(X)_{alg}$ of algebraically trivial cycles are representable. More generally, if 
$M \in \sM_{rat}(\C)$ is a motive such that $A_i(M)_{alg}$ is representable, for all $i \ge 0$, then  $M$ is finite dimensional of abelian type, see \cite{Vial 2}. \par
\noindent In particular, if $X$ is a cubic threefold in $\P^4_{\C}$, then    $h(X)$ has the following Chow-K\"unneth decomposition

$$ h(X) =\un \oplus   \L  \oplus  N \oplus \L^2   \oplus \L^3 .$$

Here   $N =h_1(J) \otimes \L= h_1(J)(1)$, with $J$   an abelian variety, isogenous to  the intermediate Jacobian $ J^2(X)$. 
Let $l$ be a general line on $X$ . Blowing up $l$ we get a conic bundle  $\tilde X \to \P^2$ whose discriminant  curve 
$C_l$ is degree 5. Let  $\pi: \tilde C_l \to C_l$ be the double cover parametrizing irreducible components of singular conics. Then we have 

$$ J^2(X) \simeq \Prym(\tilde C_l/C_l ).$$

where  $\Prym(\tilde C_l/C_l)$ is the Prym variety of $\tilde C $ over $C$, i.e.the identity component  of the fixed locus of the involution $\tau = -\sigma$ on the Jacobian variety $\Jac (\tilde C)$. Here  $\sigma$ is  the involution on $ \tilde C$ induced by the double cover. The {\it Prym motive} associated  to an \'etale  double cover  of curves $\tilde C \to C$ is the motive $(\tilde C,\pi)$, where $\pi$ is the correspondence $\pi =(id -\sigma)/2\in A^1 (\tilde C \times \tilde C)$, see \cite{NS}. Then 
$$\Delta_{\tilde C} =\pi + (\Delta_{\tilde C}+\sigma)/2 \in A^1( \tilde C \times \tilde C)$$

If  $h(\tilde C) =\un \oplus h_1(\tilde C) \oplus \L$ is  a C-K decomposition there is an isomorphism
$$\phi_{\tilde C} : A^1(\tilde C \times \tilde C)/\sI(\tilde C) \simeq \End_{\sM_{rat}}(h_1(\tilde C)) \simeq \End_{Ab}(\Jac \tilde C)\otimes \Q.$$

Here $\sI(\tilde C)$ is the ideal of degenerate correspondence, that is generated by  $[\tilde C\times P] $ and $ [P \times \tilde C]$, with $P$ a closed point and 
$ \End_{Ab}(\Jac \tilde C)$ is the group of endomorphisms as an Abelian variety, see \cite[7.4.4]{KMP}. Therefore, under the map $\phi_{\tilde C}$, the Prym motive  $\Prym(\tilde C/ C)$ corresponds to  the submotive $h_1(\tilde C)^{-}$ of $h_1( \tilde C)$ where the involution $\sigma$ acts  as -1. \par
\noindent  As  proved by Clemens and Griffiths,  $X$ is not rational, because $J^2(X)$ is not the Jacobian  variety of a curve $D$ and hence the Prym motive is not isomorphic to the mid-motive $h_1(D)(1)$.
\medskip

Let $X$ be  a cubic fourfold in $\P^5_{\C}$. In Section 2, we show that the motive $h(X)$ has a reduced  Chow-K\"unneth decomposition as follows  

$$h(X) =\un \oplus  \L  \oplus (\L^2)^{\oplus \rho_2}   \oplus  t(X)  \oplus \L^3   \oplus \L^4 $$

where $\rho_2 $ is the rank of $A^2(X)$ and all the summands of $h(X)$, but possibly $t(X)$, are finite dimensional, see  (\ref{CK}).  The motive $t(X)$ is the {\it transcendental motive } of $X$ and 
\begin{equation}A_1(X)_{hom} =A_1(X)_{alg}=A_1(t(X))\end{equation}
\noindent If $l$ is a general line on  $X$ then the surface $S_l \subset F(X)$ of  lines meeting $l$ is  smooth, with $q(S_l)=0$, and geometric genus  $p_g(S_l)=5$ ( see \cite[Sect. 3]{Vois 1}). Let $\pi: \tilde X \to \P^3$ be the conic bundle obtained by blowing up $X$ along the line $l$. The surface $S_l$ parametrizes irreducible components of the conics fibers of $\pi$. There is  an involution $\sigma$  on $S_l$ with 16 isolated fixed points and the quotient $Y_l=S_l/\sigma$ is a quintic surface in $\P^3$. The involution $\sigma$ induces a double cover $S_l \to Y_l$. Similarly to the case of the Prym motive associated to a conic bundle  one can define the Prym motive  

$$\Pr(S_l,\sigma) := t_2(S_l)^{-},$$

where $t_2(S_l)^{-}$ is the direct summand of the transcendental motive $t_2(S_l)$ where $\sigma$ acts as -1. In Prop. \ref{summand} we prove that 

$$t_2(S_l)^{-}(1) \simeq t(X)$$

Therefore, similarly to the case of a cubic  3-fold, it is natural to ask if there is  a surface $Z$ such that the Prym motive $\Pr(S_l,\sigma)$ is isomorphic to the (twisted) transcendental motive of $Z$ and hence $t(X) \simeq t_2(Z)(1)$. If $X$ is very general and therefore conjecturally not  rational, then this cannot happen, see Prop. \ref{superfici} (i). On the other hand if $X$ is special, then, assuming the Hodge conjecture and Kimura's conjecture, there exists a K3 surface $S$ such that $t_2(S)(1) \simeq t(X)$, see Remark 3.7\par
In Sect. 2 we relate the transcendental motive of $X$ with the motive of its Fano variety of lines $F(X)$ ($F$ for short).

\begin{thm}\label{1.2}
Let $h(F)$ be the motive of $F(X)$, endowed with a Chow-K\"unneth decomposition, and let $h_2(F) \cong h_2^{alg}\oplus h_2^{tr}$ be the standard decomposition of $h_2(F)$. Then we have an isomorphism
$$h_2^{tr}(F)(1)\cong h_4^{tr}(X) = t(X).$$
and therefore
 $$\Mot(X) =\Mot(F),$$
where, for a smooth projective variety $Y$ we denote by $\Mot(Y)$ the full pseudo-abelian tensor subcategory of $\sM_{rat}(\C)$ generated by $h(Y)$ and the Lefschetz motive $\L$. 
\end{thm}

 \noindent In some cases, we can say even more (see Sect. 3).
\begin{thm}\label{1.3}
Suppose $F(X)\cong S^{[2]}$, with $S$ a K3 surface. Then there is an isomorphism of motives

$$t_2(S)(1)\cong t(X).$$
and hence
$$\Mot(X) =\Mot(F)= \Mot(S)$$
\end{thm}

In particular $X$ has a finite dimensional motive if and only if the motive of $S$ is finite dimensional, in which case the transcendental motives of $X$ and $S$ are both indecomposable. Note that, according to  Kimura's conjecture and a conjecture by Y. Andr\'e (see \cite{An}),  the motives of a  cubic fourfold and of a K3 surface should be  of abelian type. \par 

\smallskip

\textbf{Added in proof:} Some time after a first version of this paper appeared, T-H B\"ulles in \cite{Bull} has given a different proof of Thm.\ref{1.2} and Thm. \ref{1.3}. His results also show that the isomorphism in Thm. 1.3  holds even when $F(X)$ is just birational to $S^{[2]}$, see \cite[Prop. 1.4]{Bull}.

\smallskip

Let  $\sC_d$ be the Noether-Lefschetz divisor of special cubic fourfold of discriminant $d$, as defined by B.Hasset in \cite{Has 1}. If $d$ satisfies the condition \par

\smallskip

(**)  $d$ is not divisible by 4,9 or a prime $p \equiv 2 (3)$ \

\smallskip

and $X$ is a general member of $\sC_d$ then, according to a conjecture of Kuznetsov \cite{Kuz} and results of Addington-Thomas \cite{AT}, $X$ should be rational. In Sect. 4 we prove that, assuming Kimura's conjecture, if  $X \in \sC_d$, with $d$ satisfying (**), then there is an isomorphism $t_2(S)(1) \simeq t(X)$, where $S$ is a K3 surface.\par
 \noindent In Section 5, by adapting some  results by Vial \cite{Vial 1} about motives of fibrations with rational fibers, we showcase classes of cubic fourfolds with a K3 surface $S$ such that $t(X) \simeq t_2(S)(1)$. More precisely we consider special cubic fourfolds that are genera members either of  $\sC_8$ or of  $\sC_{18}$. In the first case the fibers are quadrics, in the second case del Pezzo surfaces of degree 6.\par
\noindent In Sect.6 we consider the case of a cubic fourfold $X$, with an  involution, and prove (see Prop. 6.5) that  there is an isomorphism $t(X) \simeq t_2(S)(1)$, where $S \subset F(X)$ is a  K3 surface.

\section {The motive of a cubic fourfold} 

In this section we give a Chow-K\"unneth decomposition of the motive $h(X)$ of a cubic fourfold and show that its   transcendental part $h^{tr}_4(X)=t(X)$ is isomorphic to the (twisted) transcendental motive $h^{tr}_2(F(X))(1)$ coming from  a suitable Chow-K\"unneth decomposition of the motive of the Fano variety of lines $F(X)$ (see Thm. \ref{fano}). Note that, by a result of R.Laterveer \cite{Lat 1}, if $h(X)$ is finite dimensional then also $h(F(X))$ is finite dimensional.
\medskip
\noindent  Then we show that  
$$A_1(X)_{hom}=A_1(X)_{alg} \simeq A_1(t(X)).$$

\medskip

 Every cubic fourfold $X$ is rationally connected and hence $CH_0(X) \simeq \Z$. Rational, algebraic and homological equivalences all coincide for cycles of codimension 2 on $X$. Hence the cycle map 
$CH^2(X) \to H^4(X,\Z)$ is  injective and $A^2(X)=CH^2(X) \otimes \Q$ is a vector subspace of dimension $\rho_2(X)$ of $H^4(X,\Q)$.  By  the results in \cite{TZ} we  have  $A_1(X)_{hom} = A_1(X)_{alg}$. Moreover homological equivalence and numerical equivalence coincide for algebraic cycles on $X$, because the standard  conjecture $D(X)$ holds true. Therefore $A_1(X)_{hom}=A_1(X)_{num}$.\par
 \noindent A cubic fourfold  $X$ has no odd cohomology and  $H^2(X,\Q) \simeq NS(X)_{\Q} \simeq A^1(X)$, because  $H^1(X ,\sO_X) =H^2(X,\sO_X)=0$. Let $\gamma\in A^1(X)$  be the class of  a hyperplane section. Then $H^2(X,\Q)= A^1(X) \simeq \Q \gamma$ and  $H^6(X,\Q)= \Q[\gamma^3/3]$. Here   
 $< \gamma^2,\gamma^2>= \gamma^4 =3$, where   $ < \ , \ >$ is the intersection form on $H^4(X,\Q)$.\par
\noindent Let  $\pi_0 =[X \times P_0], \pi_8 = [P_0 \times X]$, where $P_0$ is a closed point and 
$\pi_2 =(1/3) (\gamma^3  \times \gamma)$, $\pi_6 =\pi^t_2 = (1/3)( \gamma \times \gamma^3)$. Then 
$$h(X) \simeq \un \oplus h_2(X) \oplus h_4(X) \oplus h_6(X) \oplus \L^4$$
where  $\un \simeq(X, \pi_0)$, $\L^4 \simeq (X, \pi_8)$, $h_2(X)=(X,\pi_2)$, $h_6(X) = (X,\pi_6)$ and $h_4(X) =(X,\pi_4)$, with $\pi_4 =\Delta_X -\pi_0-  \pi_2 -\pi_6 -\pi_8$. The above decomposition of the motive $h(X)$ is in fact integral, because
$$\gamma^3 =3\vert l \vert$$
for a line $l \in F(X)$, see \cite[Lemma A3]{SV 2}.\par
 \noindent Let  $\rho_2$ be the dimension of $A^2(X)$, i.e. the rank of the algebraic part in $H^4(X)$. Choosing 2-cycles  $\{D_1, D_2\dots,D_{\rho_2} \}$  and their Poincar\'e dual cycles $\{D'_1, D'_2\dots,D' _{\rho_2} \}$ we get a splitting  
 $$h_4(X) =h^{alg}_4(X) \oplus h^{tr}_4(X)$$

 where $\pi^{alg}_4=\sum_{1\le i  \le \rho_2}[D_i,D'_i]$ and $h^{alg}_4(X) \simeq (\L^2)^{\rho_2}$. Therefore we  get a {\it refined Chow-K\"unneth decomposition } of the motive $h(X)$

  \begin {equation}\label{CK} h(X) =  \un \oplus \L \oplus (\L^2)^{\oplus \rho_2} \oplus  t(X) \oplus \L^3 \oplus  \L^4. \end{equation}

\noindent Here $t(X) =h^{tr}_4(X)=(X,\pi^{tr}_4)$  with $\pi^{tr}_4 = \Delta_X - \pi_0-\pi_2- \pi^{alg}_4-\pi_6-\pi_8$. Then $H^*(t(X)) =H^4(t(X)) =T(X)_{\Q}$ where $T(X)$ is the transcendental lattice. All the motives in (\ref{CK}), different from $t(X)$, are isomorphic to  a multiple of $\L^i$, for some $i$.  Therefore  in  the decomposition  (\ref{CK}) all motives, but possibly $t(X)$, are finite dimensional. It follows that   the motive $h(X)$ is finite dimensional if and only if $t(X)$ is  evenly finite dimensional.\par 

\begin {lm} \label{chowtra}
 Let $X$ be a cubic fourfold and let $t(X)$ be the transcendental motive in the Chow-K\"unneth decomposition (\ref{CK}). Then $A^i(t(X)) = 0$ for $i \ne 3$ and $A^3(t(X))=A_1(X)_{hom}$
\end{lm}

\begin{proof} The cubic fourfold $X$ is rationally connected and hence $A^4(X)=A_0(X)=\Q$ that implies $A^4(t(X))=0 $. Also from the Chow-K\"unneth decomposition in (\ref{CK}) we get $A^0(t(X)) =0$ and  
$A^1(t(X))= \pi^{tr}_4 A^1(X)=0$, because $A^1(h_2(X))= \pi_2(A^1(X)) =A^1(X)$. Here and in the following we will denote by $\pi_i A^j(X)$the action of a correspondence $\pi_i$ on the Chow groups.\par
\noindent We first  show that $A^2(t(X))=0$. Let $\alpha \in A^2(X)$, with $\alpha \ne 0$. Then $\alpha$ is not homologically trivial, because $A^2(X)_{hom}=0$.
$$\pi^{tr}_4 (\alpha) = \alpha -  \pi_0(\alpha) -\pi_2(\alpha) -\pi^{alg}_4(\alpha) - \pi_6(\alpha) - \pi_8(\alpha),$$
where $\pi_0(\alpha) = \pi_8(\alpha)=0$. We also have
$$ \pi_2(\alpha) = (1/3)[\gamma^3 \times \gamma]_*(\alpha)= (1/3)(p_2)_*((\alpha \times X)\cdot [\gamma^3 \times \gamma])$$
where $p_2 :X \times X \to X$ and $\gamma^3 \in A^3(X)$. Therefore $\pi_2(\alpha) =0$ in $A^2(X)$. Similarly 
$$\pi_6(\alpha) =  (1/3)[\gamma \times \gamma^3]_*(\alpha)= (1/3)(p_2)_*((\alpha \times X) \cdot  [\gamma \times \gamma^3]) $$
where $\alpha \cdot \gamma \in A_1(X)/A_1(X)_{hom} \simeq \Q[\gamma^3/3]$ and hence $\alpha \cdot \gamma= (a/3) [\gamma^3]$ with $a \in \Q$. Therefore
 $\pi_6(\alpha) =0$ in $A^2(X)$. Let $\{D_1,\dots, D_{\rho_2}\}$ be  a $\Q$-basis for $ A^2(X)$ and let $\alpha =\sum_{ 1\le i \le \rho_2} m_i D_i$, with $m_i \in \Q$.
Then $\pi^{alg}_4(\alpha) = \sum_{1 \le i \le \rho_2} \pi_{4,i}(\alpha)= \alpha$, because $(\pi_{4,i})_*(D_i)=  D_i $. We get $\pi^{tr}_4 (\alpha)  =\alpha -  \pi^{alg}_4 (\alpha) =0 $ and hence
$$A^2(t(X))= (\pi^{tr}_4)_*A^2(X)=0.$$
\noindent Therefore we are left to show that $A_1(t(X))=A_1(X)_{hom}$. Let $\beta \in A_1(X) =A^3(X)$. From the Chow-K\"unneth decomposition in (\ref{CK}) we get
$$\pi^{tr}_4 (\beta) = \beta -  \pi_0(\beta) -\pi_2(\beta) -\pi^{alg}_4(\beta) - \pi_6(\beta) - \pi_8(\beta),$$
where $\pi_0(\beta) = \pi^{alg}_4(\beta) =\pi_8(\beta)=0$. We also have $\pi_2(\beta) = 0$ because $ \pi_2 =(1/3)(\gamma^3 \times \gamma)$. Therefore
$$\pi^{tr}_4(\beta) =\beta - \pi_6(\beta) = \beta -(1/3)(\gamma \times\gamma^3)_*(\beta) = \beta -(1/3) (\beta \cdot \gamma)\gamma^3 \in A^3(X) $$
and hence 
$$(\pi^{tr}_4(\beta)\cdot \gamma) = (\beta \cdot \gamma) -  (1/3) (\beta \cdot  \gamma)(\gamma^3\cdot \gamma)=0$$ because $\gamma^4 =3$.
 Since $\gamma$ is a generator of $A^1(X)$ it follows that  the cycle   $\pi^{tr}_4(\beta)$ is  numerically trivial. Therefore we get
$$A_1(t(X))=\pi^{tr}_4A_1(X)=A_1(X)_{num}=A_1(X)_{hom}.$$
\end{proof}

The following Lemma follows from the results in \cite[Thm. 3.18]{Vial 1} and \cite[Lemma 1]{GG}.

\begin{lm}\label{morphism} Let $f : M \to N$ be a morphism of motives in $\sM_{rat}(\C)$ such that $f_*:  A^i(M) \to A^i(N)$ is an isomorphism for all $i \ge0$. Then $f $ is an isomorphism.\end{lm}

\begin{proof} Let $M=(X,p,m)$  and  $N =(Y,q,n)$ and let  $k \subset \C$ be a field of definition of  $f$, which is finitely generated . Then $\Omega=\C$ is a universal domain over $k$. By \cite[Thm. 3.18]{Vial 1} the map $f$ has a right inverse, because the map $f_*: A^i(M)\to A^i(N)$ is surjective.  Let  $g : N \to M$ be such that $f \circ g =id_N$.  Then $g$ has an image $T$ which  is a direct factor of $M$ and hence $f$ induces an isomorphism  of motives in $\sM_{rat}(\C)$ 
$$f :  M \simeq N \oplus T$$
From the isomorphism  $A^i(M)\simeq A^i(N)$, for all  $i \ge 0$, we get $A^i(T)=0$ and hence $T=0$, by \cite[Lemma 1]{GG}.\end{proof} 
Let $X$ be a cubic fourfold and let $F(X)=F $ be its Fano variety of lines , which is a smooth fourfold. Let
\begin {equation}\label{incidence}\CD  P@>{q}>> X \\
@V{p}VV   @.         \\
F
\endCD \end{equation}
be the incidence diagram, where $P \subset X \times F$ is the universal line over $X$.    Let   $  p_*q^* : H^4(X,\Z) \to H^2(F ,\Z)$ be the Abel-Jacobi map.  Let $\alpha_1,\dots ,\alpha_{23}$ be a basis of    $H^4(X,\Z)$ and  let  $\tilde \alpha_i =p_*q^*(\alpha_i)$.  Then, by a result of Beauville-Donagi in \cite{BD},  $\tilde \alpha_i =p_*q^*(\alpha_i)$ form a basis of $H^2(F,\Z)$.  The lattice $H^2(F,\Z)$ is endowed with the   Beauville-Bogomolov bilinear form $q_F$, see \cite[Sect. 19]{SV 2}. The Abel-Jacobi map  induces an isomorphism between the primitive cohomology of  $H^4(X,\Z)_{prim}$ and the primitive cohomology $H^2(F,\Z)_{prim}$. Here $H^2(F,\Z)_{prim} =<g>^{\perp}$, with $g \in H^2(F,\Z)$ the restriction to $F(X) \subset \Gr (2,6)$ of the class on $\Gr(2,6)$ defining the Pl\"ucker embedding. In particular $g =p_*q^*(\gamma^2)$. The Abel-Jacobi map induces an isomorphism between the Hodge structure of $H^4(X,\C)_{prim}$ and the (shifted) Hodge structure of $H^2(F,\C)_{prim}$.\par
\noindent The  next result shows that the Abel-Jacobi map induces an isomorphism between  $t(X)$  and the transcendental motive $h^{tr}_2(F)$ in a suitable Chow-K\"unneth decomposition for $h(F)$.

\begin{thm}\label{fano} Let $X$ be cubic fourfold and let $F(X)$ be  its Fano variety  of lines. Then there exists a Chow-K\"unneth decomposition 
$$h(F) =h_0(F) \oplus h_2(F) \oplus h_4(F) \oplus h_6(F) \oplus h_8(F)$$
with  $h_2(F) \simeq h^{alg}_2(F) \oplus h^{tr}_2(F)$. The Abel-Jacobi map gives an isomorphism  
$$h^{tr}_2(F) (1)\simeq h^{tr}_4(X) =t(X).$$ \end{thm}

\begin {proof} The hyperk\"ahler manifold $F(X)$ is of $K3^2$-type, \it i.e. \rm it is deformation equivalent to the Hilbert scheme of length-2 subschemes on a K3 surface. By the results in \cite[Sect. 19]{SV 2}, there exists a cycle $L \in \CH^2(F \times F)$ whose cohomology class in $H^4(F \times F,\Q)$ is the Beauville-Bogomolov class 
$\sB$, i.e. the class corresponding to $q^{-1}_F$.  Let us set $l := (i_{\Delta})^*L \in CH^2(F)$, where $i_{\Delta} : F \to F \times F$ is the diagonal embedding. By \cite[Thm. 2]{SV 2}, the Chow groups of the variety $F$ have a {\it Fourier decomposition}. In particular the group   $A^4(F) =A_0(F)$ has a  canonical decomposition  

$$A^4(F) = A^4(F)_0 \oplus A^4(F)_2 \oplus A^4(F)_4$$

with  $A^4(F)_0 = <l^2>$, $ A^4(F)_2 = l \cdot L_*A^4(F)$  and $ A^4(F)_4 = L_*A^4(F)   \cdot  L_*A^4(F)$.  Here   $<l^2> = \Q c_F$, with  $c_F$  a special degree 1 cycle coming from a surface $W \subset F$ such that any two points on $W$ are rationally equivalent on $F$, see \cite[Lemma A.3]{SV 2}.\par
 The Fourier decomposition of the Chow groups $A^*(F)$ is compatible with  a Chow-K\"unneth decomposition of the motive $h(F)$ given by projectors 
 $$\{\pi_0(F),\pi_2(F),\pi_4(F),\pi_6(F),\pi_8(F)\},$$
as in \cite[Thm. 8.4]{SV 1}. Here $\pi_2(F) =\pi^{alg}_2 \oplus \pi^{tr}_2, \pi_6(F) =\pi^{alg}_6 \oplus \pi^{tr}_6$ and 
$$\pi_4 =\Delta_F - (\pi_0 -\pi_2 -\pi_6 -\pi_8)$$
\noindent  We have $h(F) = M \oplus N$ where 
$$N= (F,\pi_0) \oplus (F,\pi^{alg}_2) \oplus (F,\pi^{alg}_4) \oplus (F,\pi^{alg}_6) \oplus (F,\pi_8)$$
and $N$ is isomorphic to a direct sum of $\L^i$, for $i \ge 0$. We also have  $A^*(F)_{hom}=A^*(M)$. Let us set $h_2(F) =h^{alg}_2(F) \oplus h^{tr}_2(F)$,  where $h^{tr}_2(F) =(F,\pi^{tr}_2(F))$, with   $\pi^{tr}_2(F) \in \End_{\sM_{rat}}M$ and $H^*(h^{tr}_2(F)) =H^2_{tr}(F)$. Then
$$A^2(F) = \Im(\pi_4)_* \oplus \Im (\pi_2)_*=  \Im(\pi_4)_* \oplus  \Im(\pi^{tr}_2)_* $$
\noindent because $\pi^{alg}_2(F)$ acts as 0 on $A^2(F)$.\par
\noindent   Let us denote $\sA = I_*A^4(F) \subset A^2(F)$, with $I$ the incidence correspondence, i.e. $I =(p\times p)_*(q\times q)^*\Delta_X$ . The group $\sA_{hom}$ is generated by the classes $[S_{l_1}]- [S_{l_2}]$ where, for a line $l$ on $X$, $S_l$ denotes the surface in $F(X)$ of all lines meeting $l$, see \cite[Thm. 21.9]{SV 2}. By \cite[21.10]{SV 2} the group $\sA_{hom}$ coincides with the subgroup $ A^2(F)_2$ in the Fourier decomposition $A^2(F) =A^2(F)_0 \oplus A^2(F)_2$.\par
\noindent  The Abel-Jacobi map $q_*p^*: A^i(F) \to A^{i-1}(X)$ induces a surjective map $\Psi_0 : A^4(F) \to A^3(X) =A_1(X)$,  where $A_1(X)$ is generated by the classes of lines, see \cite{TZ}. The map induced by $\Psi_0$ on the subgroup $A^4(F)_{hom}$ has a kernel isomorphic to $F^4A^4(F)=\sA_{hom} \otimes  \sA_{hom}$, see \cite[Thm 20.2]{SV 2}, where  $F^4A^4(F) =\Ker \{ I_* : A^4(F) \to A^2(F)\}$. The maps $I_*$ and $\Psi_0$ yield two exact sequences

$$\CD   0@>>>  F^4A^4(F)@>>> A^4(F)_{hom}@>{I_*}>>\sA_{hom}@>>> 0 \endCD $$

$$\CD   0@>>>  F^4A^4(F)@>>> A^4(F)_{hom}@>{\Psi_0}>>A_1(X)_{hom}@>>> 0 \endCD $$
where $ (A^4(F))_{hom} = (A^4(F)_2)_{hom} \oplus  (A^4(F)_4)_{hom}$, with   $(A^4(F)_2)_{hom} \simeq \sA_{hom}$ and $(A^4(F)_4)_{hom} \simeq \sA_{hom} \cdot \sA_{hom}$. Therefore we get the following isomorphisms

$$ \sA_{hom}\simeq A^4(F)_2\simeq  A_1(X)_{hom} ;$$
 $$  \sA_{hom} \simeq A^2(F)_2  \simeq  A_1(X)_{hom}.$$
 
By \cite[Proposition 7.7]{SV 1} we also have 
$$A^2(F)_{hom} \simeq \sA_{hom} \Longleftrightarrow \Im (\pi_2)_* = A^2(F)_{hom}.$$
Therefore  $A^2(F)_{hom}=\Im (\pi^{tr}_2)$  and we get an isomorphism
$$A^2(h^{tr}_2(F)) \simeq A_1(X)_{hom}.$$
The universal line $P$, viewed as a correspondence in $A_5(F \times X)$, yields a map in $\sM_{rat}(\C)$ 
$$P_* : h(F)(1)\to h(X).$$
By the results in \cite{SV 2} the relation between the Chow groups of $F$ and $X$ is given  via $P$. Therefore, by composing with the projection $h(X) \to t(X)$ and the inclusion $h^{tr}_2(F)(1) \subset h(F)(1)$, the correspondence $P$ yields  a map of motives 
$$\bar P_* : h^{tr}_2(F)(1) \to t(X)$$

The above map induces a map of Chow groups
$$A^i(h^{tr}_2(F)(1))\to  A^i(t(X))$$
that is an isomorphism for all$i \ge0$ because 
$$A^3(h^{tr}_2(F)(1)) =A^2(h^{tr}_2(F))\simeq A_1(X)_{hom} = A^3(t(X))$$
and $A^i(h^{tr}_2(F)) =A^i(t(X)=0$ for $i \ne 3$. By Lemma \ref{morphism} we get $h^{tr}_2(F)(1)\simeq t(X)$.\end{proof}
\begin{rk}\label{finiteness} If the motive $h(F(X))$ is finite dimensional then, by Theorem \ref{fano}, also $t(X)$ is finite dimensional and hence $h(X)$ is finite dimensional. Conversely if $h(X)$ is finite dimensional then, by \cite{Lat 1}, also $h(F(X))$ is finite dimensional .\end{rk}
\medskip 
  Let $X$ be a cubic fourfold and  let $l \in F(X)$ be a general line. There exists a unique plane $P_l \subset \P^5$  containing $l$ and which is everywhere tangent to $X$ along $l$. 
Then
$$P_l \cdot X =2[l] + [l_0]$$
Let $\phi : F \dashrightarrow F$ be the rational map  defined by C.Voisin in \cite{Vois 2}. The map $\phi$ sends a general line $l \subset X$ to its residual line with respect  to the unique plane $\P^2 \subset \P^5$ tangent  to $X$ along $l$. Let $S_l \subset F$ be the surface of lines meeting a general line $l$. Then $S_l$ is a smooth surface  with $q(S_l)=0$ and   $p_g(S_l) =5$ ( see \cite[Lemma 1 and 3]{Vois 1}) and there is a natural involution $\sigma: S_l \to S_l$. If $[l'] \in S_l$ is  a point different from $[l]$ then $\sigma([l']) $ is the residue line of $l \cup l'$, while $\sigma ([l] )= [l_0]$. The involution $\sigma$ has 16 isolated fixed points and the quotient $Y_l =S_l/\sigma$ is a quintic surface in $\P^3$ with 16 ordinary double points, see \cite[Remark 4.4]{Shen}. Let $\tilde X$ be the blow-up of $X$ along $l$. Then the projection from the line $l$  defines a conic bundle  $\pi: \tilde X \to \P^3$. The surface $S_l$ parametrizes lines in the singular fibers,  the discriminant divisor $D \subset \P^3$ is the quintic surface $Y_l$ and the induced map $S_l \to D$ is the double cover $f_l: S_l \to Y_l$ associated to the involution $\sigma$.
The map $f_l :S_l \to Y_l$ induces a commutative diagram

$$\CD \tilde S_l@>{g_l}>>S_l \\
@V{\bar f_l}VV    @V{f_l}VV \\
\tilde Y_l@>>> Y_l\endCD $$

where $\tilde S_l$ is the blow-up of the set of isolated fixed points of $\sigma$ and $\tilde Y_l$ is a desingularization of $Y_l$. For a general $l$ the quintic surface $Y_l$ is normal with rational singularities, hence it has  $p_g(Y_l)=4$ and $q(Y_l)=0$, see \cite{Yang}. Since $t_2( - )$ is a birational invariant for smooth projective surfaces the above diagram yields a map 
$$\theta : t_2 (\tilde S_l)=t_2(S_l) \to t_2(\tilde Y_l)$$
which is a projection onto a direct  summand.   Since $q(S_l)=0$ the motive $t_2(S_l)$ splits as follows,  see \cite[Prop. 1]{Ped}

$$t_2(S_l) \simeq t_2(S_l)^{+} \oplus t_2(S_l)^{-}$$

where  $t_2(S_l)^{+}=t_2(\tilde Y_l)$ and $t_2(S_l)^{-}$ are the direct summand of $t_2(S_l)$ on which the involution $\sigma$ acts  as $+1$ and $-1$ respectively . We also have $t_2(\tilde Y_l)\ne 0$, because $p_g(Y_l) \ne 0$ and 

$$A^2(t_2(\tilde Y_l))= A^2(t_2(S_l))^{+}   = A_0(S_l)^{+}_0  \ ;\ \ \ \   A^2(t_2(S_l))^{-}  = A_0(S_l)^{-}_0;$$

where $A_0(S_l)_0 = A_0(S_l)^{+}_0 \oplus A_0(S_l)^{-}_0$.  Here $A_0(S_l)_0$ is the group of 0-cycles of degree 0 (with $\Q$-coefficients) and $A_0(S_l)^{+}_0$ is the subgroup  fixed by $\sigma$.\par
Let $\sC_l $ be the total space of lines meeting $l$ and let
$$\CD  \sC_l@>{q_l}>> X \\
@V{p_l}VV   @.         \\
S_l \subset F
\endCD $$
be the incidence diagram. Then the Abel-Jacobi map $\Phi$ and the cylinder homomorphism $\Psi$  induce
$$\Phi_l : A_i(X) \to A_{i-1}(S_l) \  ;  \  \Psi_l : A_i(S_l) \to A_{i+1}(X).$$

The following result shows the relation  between the transcendental motive $t(X)$ and the transcendental motive  of the surface $S_l$.

\begin {prop} \label{summand} 
Let $X$ be a  cubic fourfold and let $S_l$ be the surface of lines meeting a general line $l \subset X$. Then\par
(i) $t(X) \simeq t_2(S_l)(1)^{-}$\par
(ii) $\Phi_l$ induces an isomorphism: $H^4_{tr}(X,\Q) \simeq H^2(S_l,\Q)^{-}$, with  $H^2(S_l,\Q)^{-}$  the subgroup where $\sigma$ acts as -1.
 \end{prop}

\begin {proof} \bf (i) \rm  By \cite[Thm. 4.7]{Shen} the composition $\Phi_l \circ \Psi_l$ equals $\sigma - id$ and $\Psi_l \circ \Phi_l =- 2$. The Abel-Jacobi map $\Phi_l$ induces an isomorphism between $A_1(X)$ and  $\Pr(A_0(S_l)_0,\sigma)$, where $\Pr(A_0(S_l)_0,\sigma)= A_0(S_l)_0^{-}= A^3(t_2(S_l)^{-}(1))$ \cite[Def. 3.6]{Shen}. Therefore the  map 
$$\Phi_l :  A_1(X) \simeq \Q[\gamma^3/3]\oplus A_1(X)_{hom}  \to A_0(S_l)_0$$ 
yields an isomorphism between $A_1(X)_{hom}$ and $A_0(S_l)_0^{-}$.\par
\noindent  In the  incidence diagram $\sC_l$ is a $\P^1$-bundle over $S_l$ and hence $h(\sC_l) \simeq h(S_l) \oplus h(S_l)(1)$. Since $q(S_l)=0$ we the motive $hS_)$ has a  C-K decomposition $h(S_l)=\un\oplus \L^{\otimes\rho}\oplus t_2(S_l) \oplus \L^2$
and hence we get a map 
 $$ g : t_2(S_l)(1) \to t(X),$$
where  $t_2(S_l)\simeq t_2(S_l)^{+} \oplus t_2(S_l)^{-}$. The map $g$ when restricted to  $t_2(S_l)^{-}(1)$ gives
$$(g^{-})_* : t_2(S_l)(1)^{-} \to t(X),$$
such that the induced map on Chow groups
$$(g)_* : A^i(t_2(S_l)^{-}(1)) \to A^i(t(X))$$
 is an isomorphism for all $i\ge0$, because  $A^3(t_2(S_l)^{-}(1))= A_0(S_l)_0^{-} $, $A^3(t(X)) = A_1(X)_{hom}$.
while $A^i(t_2(S_l)^{-}(1)) =0$ and  $A^i(t(X))=0$ for $i \ne 3$ .Therefore  $g^{-}$ is an isomorphism in $\sM_{rat}(\C)$, see Lemma \ref{morphism}.\par
\bf (ii) \rm  In \cite[Cor. 4.8]{Shen} it is proved that the Abel-Jacobi map $\Phi_l : H^4_{tr}(X, \Z) \to H^2(S_l,\Z)^{-}$ is an isomorphism of lattices.
\end{proof}

\begin {rk}  (1) The isomorphism in (i) answers a question raised by M.Shen in a private communication. For a smooth projective surface $S$, with $q(S)=0$ and $p_g(S) >0$, equipped with an involution $\sigma$,  we can define the {\it  Prym motive $ \Pr(S,\sigma)$} to be the motive 
$$  \Pr(S,\sigma) =t_2(S)^{-}$$
where, as in \cite[Prop 1]{Ped}, $t_2(S)^{-}$ is the direct summand of $t_2(S)$ where the involution $ \sigma $ acts as $-1$. The action of $\sigma$ on $t_2(S)$ is defined via the homomorphism
$$ \Psi_S : A^2(S \times S) \to \End_{\sM_{rat}}(t_2(S))$$
which sends  the correspondence $\Gamma_{\sigma} \in A^2(S \times S)$ to $\pi^{tr}_2  \circ \Gamma_{\sigma} \circ \pi^{tr}_2 $. Here  $t_2(S) = (S,\pi^{tr}_2)$ and hence the projector $\pi^{tr}_2$ corresponds to the identity in  $\End_{\sM_{rat}}(t_2(S))$.\par
 (2) If $l \in X$ is a general line then the blow-up $\tilde X$ of $X$ along $l$ is a conic bundle $\pi : \tilde X \to \P^3$ and $Y_l =S_l/\sigma$ is the discriminant divisor. 
The involution $\sigma$ on $S_l$ is given by the double cover $S_l \to Y_l$. Therefore  (ii)   may be viewed as a generalization  of a result appearing in \cite{NS} for a conic bundle $f: X \to \P^2$. In \cite{NS} it is proved that the motive of $h(X)$ is determined by the Prym motive 
$\Pr(\tilde C/C)$, where the curve $C$ is the discriminant of $f$ and $\tilde C \to C$ is the usual double cover.
\end{rk}

\begin{rk}\label{k3prym}
In the case $p_g(S) =1$ (e.g.  $S$ a K3 surface) then $t_2(S)$ is indecomposable if $h(S)$ is finite dimensional (see \cite[Cor 3.10]{Vois 3}). Therefore the Prym motive of $S$ is either 0 or it coincides with $t_2(S)$. If $S$ is a K3 surface and  $\sigma$ is a symplectic involution then $t_2(S) \simeq t_2(S/\sigma)$ and hence $\sigma$ acts as the identity on $t_2(S)$, i.e $\Pr(S,\sigma)=0$. If $\sigma$ is non-symplectic then the quotient surface $S/\sigma$ is either an Enriques surface or a rational surface. In any case $t_2(S/\sigma) =0$ and 
$\Psi_S(\Gamma_{\sigma}) =- id_{t_2(S)}$. Therefore $\Pr(S,\sigma) =t_2(S)$. 
\end{rk}

\section { Special cubic fourfolds}
 In this section we prove  (see Thm.\ref{mapiso}) that, if 
 $F(X)$ is isomorphic to $S^{[2]}$, with $S$ a K3 surface, then $t(X)$ is isomorphic to $t_2(S)(1)$. Therefore $h(X)$ is finite dimensional if and only if $h(S)$ is finite dimensional.\par
 \noindent Recall that a cubic fourfold $X$ is \it special \rm if it contains a surface $Z$ such that its cohomological class $\zeta$ in $H^4(X,\Z)$ is not homologous to any multiple of  $\gamma^2$. Therefore $\rho_2(X) >1$. The discriminant $d$ is defined as the discriminant  of the intersection form $< ,>_D$ on the sublattice $D$ of $H^4(X,\Z)$ generated by $\zeta$ and $\gamma^2$. B.Hassett in \cite{Has 1} proved that special cubic fourfolds  of discriminant $d$ form an irreducible divisor $\sC_d $ in the moduli space $\sC$ of cubic fourfolds if and only if  $d>0$ and $d\equiv 0,2 (6)$.\par 
 
\begin{defn}\label{associated} Let $X$ be a special cubic fourfold and let  $D$ be the sublattice of $H^4(X,\Z)$ generated by $\zeta$ and $\gamma^2$. A polarized K3 surface $S$ is {\it associated} to $X$ if there is an isomorphism of lattices $ K \simeq H^2(S,\Z)_{prim}(-1)$, where  $K =D^{\perp}$ and $H^2(S,\Z)_{prim} $ denotes primitive cohomology with respect to a polarization $l\in H^2(S,\Z)$.\end{defn}

\noindent  If $X$ is a generic special cubic fourfold with discriminant of the form $d =2(n^2 +n+1)$, where $n$ is an integer $\ge2$, then the Fano variety of $X$ is isomorphic to $S^{[2]}$, with $S$ a K3 surface associated to X. Special cubic fourfolds of discriminant $d>6$ have associated K3 surface $S$  if and only if $d$ is not divisible by 4  or  9 or any odd prime $p \equiv 2 (3)$. In this case  the transcendental lattice  $T(X)$ is Hodge isometric  to $T(S)(-1)$, see \cite{Add}.\par
\noindent In the case $d =14$  the  special surface is a smooth quartic rational normal scroll. By the results in \cite{BD}  and in \cite{BRS} all the fourfolds $X$ in $\sC_{14}$ are rational (see also \cite{ABBV} for details on the derived categories approach). Moreover if $X \in (\sC_{14} - \sC_8) $, then $F(X) \simeq S^{[2]}$, where $S$ is the K3 surface of degree 14 and genus 8, parametrizing smooth quartic rational normal scrolls contained in $X$. 

\medskip

 \noindent More generally, suppose that  $X$ is special and $F(X) \simeq S^{[2]}$, with $S$ a K3 surface. Then the homomorphism $H^2(S,\Q) \to H^2(F,\Q)$ induces an orthogonal direct sum decomposition with respect to the Beauville-Bogomolov form
$$ H^2(F,\Q) \simeq H^2(S,\Q) \oplus \Q \delta,$$
with $q_F(\delta,\delta)=-2$ and $q_F$ restricted to $H^2(S,\Q)$ is the intersection form, see \cite[Rmk. 10.1]{SV 2}. Therefore
 $$H^4_{tr}(X,\Q) \simeq H^2_{tr}(S,\Q)$$ 
where $ \dim H^4_{tr}(X,\Q) =23 -\rho_2(X)$. Here $\rho_2(X) \ge2$ and hence we get
$$ \dim H^2_{tr}(F,\Q) =\dim H^2_{tr}(S,\Q) = 22 -\rho(S) = 23 -\rho_2(X) \le 21,$$
where $\rho(S)$ is the rank of $NS(S)$.\par

 \begin{thm}\label{mapiso}  Let $X$ be  a cubic fourfold and let $F =F(X)$ be the Fano variety of lines. Suppose that $F \simeq S^{[2]}$, with  $S$ a K3 surface.  Let  $P$ be the correspondence in the incidence diagram  (\ref{incidence}). Then $P_*$ induces a map of motives $\bar q : t_2(S)(1) \to t(X)$ in $\sM_{rat}(\C)$ which is  an isomorphism.  
\end{thm}

\begin{proof}  In (\ref{incidence}) the universal line $P$ as a correspondence in $A_5(F\times X)$ gives a map 
$$P_* : h(F)(1) \to h(X)$$
 By the results in \cite[Thm. 6.2.1]{deC-M} $h(S)$ is a direct summand of $h(S^{[2]})=h(F)$. Therefore we get a map  
$$ \CD  h(S)(1)@>>>h(F)(1) @>{P_*}>>h(X) \\ \endCD$$
  Let 
$$h(S) \simeq \un \oplus \L^{\oplus \rho(S)} \oplus t_2(S) \oplus \L^2$$
be a refined Chow-K\"unneth decomposition, as in \cite[Sect. 7.2.2]{KMP}.  By composing with the inclusion $t_2(S)(1) \to h(S)(1)$ and the surjection $h(X) \to t(X)$ we get  a map of motives  in  $\sM_{\rat }(\C)$,
$$P_* :  t_2(S)(1) \to t(X)$$ 
For two distinct points $x,y \in S$ let us denote by $[x,y] \in F =S^{[2]}$ the point of $F$ that corresponds to the subscheme $x \cup y \subset S$. If  $x=y$  then $[x,x]$ denotes the element in $A^4(F)$ represented by any point corresponding to a non reduced subscheme of length 2 on $S$ supported on $x$. With these notations the special degree 1 cycle $c_F\in A^4(F)$ (see \cite[Lemma A.3]{SV 2}), given by any point on a rational surface $W \subset F$, is represented by the point $[c_S, c_S]  \in F$, where $c_S$  is the Beauville-Voisin cycle in $A_0(S)$ such that $c_2(S) = 24  c_S$. We also have (see \cite[Prop. 15.6]{SV 2}) 
$$(A^4(F)_2 )_{hom} = <[c_S,x]- [c_S,y]>.$$
  We claim  that the map $\phi : A_0(S) \to A_0(S^{[2]})=A^4(F)$ sending $[x]$ to $[c_S, x]$ is injective and hence 
 $$A_0(S)_0  \simeq (A^4(F)_2)_{hom}.$$
  \noindent The variety $S^{[2]}$ is the blow-up of the symmetric product $S^{(2)}$ along the diagonal $\Delta \cong S$. Let $\tilde S$ be the inverse image of $\Delta$ in $S^{[2]}$. Then $\tilde S$ is the image of the closed embedding $s \to [c_S, s]$. By a result proved in \cite[Thm. 2.1]{Ba} the induced map of 0-cycles $A_0( \tilde S) \to A_0(S^{[2]})$ is injective. Therefore the map $\phi$  is injective.\par
 \noindent From the isomorphism $A_0(S)_0 \simeq (A^4(F)_2)_{hom}$ we get 
 $$A^3(t_2(S)(1) = A^2(t_2(S))=A_0(S)_0 \simeq A_1(X) _{hom}\simeq A^3(t(X))$$
 Since $A^i(t_2(S)(1)) =A^i(t(X))=0$ for $i \ne 3$ the map $\bar P  : t_2(S)(1) \to t(X)$ gives an isomorphism on all Chow groups .Therefore  $t_2(S)(1) \simeq t(X)$ 
    \end{proof}
    
\begin{rk} Let $X$ be a cubic fourfold such that there exist K3 surfaces $S_1$ and $S_2$ and  isomorphisms $r_1 : F(X) \to S^{[2]}_1$ and $r_2: F(X) \to S^{[2]}_2$ with  $r_1^*\delta_1 \ne r^*_2\delta_2$, as in \cite[Def. 6.2.1]{Has 1}, where $H^2(F,\Q)\simeq H^2(S_1,\Q) \oplus \Q \delta_1 \simeq H^2(S_2,\Q) \oplus \Q\delta_2$.
Then, by Thm. \ref{mapiso}, we get  $t_2(S_1) \simeq t_2(S_2)$, and hence the motives $h(S_1)$ and $h(S_2)$ are isomorphic. \end{rk}

\begin{cor}\label{k3finite} Let $X$ be  a cubic fourfold and let $F =F(X)$ be the Fano variety of lines. Suppose that   $F \simeq S^{[2]}$, with  $S$ a K3 surface. Then $h(X)$ is finite dimensional   if and only if   $h(S)$ is finite dimensional in which case the motive $t(X)$ is indecomposable.
 \end{cor}
 
 \begin{proof} If $h(X)$ is finite dimensional then also $t(X)$ is finite dimensional and hence, by Theorem \ref{mapiso},  $t_2(S)$ is finite dimensional. Therefore $h(S)$ is finite dimensional. Conversely, if $h(S)$ is finite dimensional then also $t_2(S)$ and $t(X)$ are finite dimensional, by Theorem \ref{mapiso}. From the Chow-K\"unneth decomposition in (\ref{CK}) we get that $h(X)$ is finite dimensional.   If $h(S)$ is finite dimensional  then the motive $t_2(S)$ is indecomposable, see \cite[Cor 3.10]{Vois 3}, and hence also $t(X)$ is indecomposable.  \end{proof}
 
\begin {rk}   If the motive $h(X)$  of a cubic fourfold is finite dimensional then  the transcendental part $t(X)$  of $h(X)$ is, up to isomorphisms in $\sM_{rat}(\C)$, independent of the Chow-K\"unneth decomposition  $h(X) =\sum_i h_i(X)$ in (\ref{CK}). If   $h(X) =\sum_i \tilde h_i(X)$ is another Chow-K\"unneth decomposition, with $\tilde h_i(X) =(X,\tilde \pi_i)$, then, by \cite[Thm. 7.6.9]{KMP}, there is an isomorphism $\tilde h_i(X) \simeq h_i(X)$ and 
$\tilde  \pi_i =(1+Z)\circ \pi_i \circ (1+Z)^{-1}$, where $Z \in A^4(X \times X)_{hom}$ is a nilpotent correspondence. In particular 
$$\tilde \pi_4= (1+Z)\circ \pi_4\circ (1+Z)^{-1}=(1+Z)\circ (\pi^{alg}_4+\pi^{tr}_4) \circ (1+Z)^{-1}$$
and hence $\tilde h_4(X)$ contains as a direct summand a submotive $\tilde t(X) =(X, (1+Z)\circ \pi^{tr}_4 \circ (1+Z)^{-1})$ isomorphic to $t(X)$.\par
\noindent However, differently from the case of the transcendental motive $t_2(S)$ of a surface $S$, the motive $t(X)$ is not a birational invariant. In fact $t(X)\ne 0$ for a rational cubic fourfold $X$ such that $F(X) \simeq S^{[2]}$, with $S$ a K3 surface, while $\P^4_{\C}$ has no transcendental motive.\par
\end{rk}

According to Cor. \ref{k3finite}, if $X$ is a special cubic fourfold with $F(X) \simeq S^{[2]}$, and $h(X)$ is  finite dimensional, then $t(X)$ is indecomposable. The following proposition shows that, if $X$ is not special  and $h(X)$ is  finite dimensional, then  $t(X)$ is indecomposable.

\begin{prop}\label{superfici} Let $X$ be a very general cubic fourfold, i.e. $\rho_2(X)=1$. Then\par

(i) The transcendental motive $t(X)$ is not isomorphic to $t_2(S)(1)$, for a smooth projective surface $S$;

(ii) If  $h(X)$ is finite dimensional  $t(X)$ is indecomposable.\end{prop}

\begin {proof} (i) Suppose that there exists a smooth projective surface $S$ such that $t(X) \simeq t_2(S)(1)$. Then

$$H^4(t(X))=H^4_{tr}(X,\Q) \simeq H^4(t_2(S)(1)) =H^2_{tr}(S,\Q).$$

Since  $h^{3,1}(X)=h^{1,3}(X)=1$, we get $h^{2,0}(S)=h^{0,2}(S)=1$ and therefore the surface $S$ has geometric genus $p_g(S)=1$. By the results in \cite[Sect. 2]{Mo} there exists a K3 surface $\tilde S$ such that $H^2_{tr}(S,\Q) \simeq H^2_{tr}(\tilde S,\Q)$. Since the dimension of $H^2_{tr}(\tilde S,\Q)$ is at most 21 we should also have $\dim H^4_{tr}(X,\Q) \le 21$, while, for a very general cubic fourfold , $\dim H^{tr}(X,\Q) =23 -1 =22$.\par
\noindent (ii)  Let us define the primitive motive $h(X)_{prim}=( X,\pi_{prim},0)$ as in \cite[Sect. 8.4]{Ki}, where
$$ \pi_{prim} = \Delta_X - (1/3)\sum_{0 \le i \le 4}( \gamma^ {4-i} \times \gamma^i).$$
and
$$ H^*(h(X)_{prim})= H^4(X,\Q)_{prim} .$$
\noindent Since  $X$ is very general, $\rho_2(X)=1$ and $ A^2(X)$ is generated by the class $\gamma^2$. Therefore in the Chow-K\"unneth decomposition of $h(X)$ in (\ref{CK})  we have
$h(X)_{prim} = h^{tr}_4(X) =t(X)$ and
$$h_4(X) =h^{alg}_4(X) +h^{tr}_4(X) \simeq \L \oplus h(X)_{prim}.$$
Let $\sM_{hom}(\C)$ be the category of homological motives and let $\widetilde \sM_{hom}(\C)$ be the subcategory generated by the motives of all smooth projective varieties  $V$ such that the K\"unneth components of the diagonal in $H^*(V \times V)$ are algebraic.  The Hodge realization functor
$$H_{Hodge} : \sM_{rat}(\C) \to HS_{\Q}$$
to the Tannakian category of $\Q$-Hodge structures induces a faithful functor $\widetilde \sM_{hom}(\C) \to HS_{\Q}$. Let us denote $\bar h(X):= h^{hom}(X) \in \widetilde \sM_{hom}(\C)$. Since  $X$ is very general   $\End_{HS} (H^4(X,\Q)_{prim} )=\Q[id]$, see \cite[Lemma 5.1]{Vois 2}. Therefore 
$\End_{\sM_{hom}}(\bar h(X)_{prim}) \simeq \Q[id]$ and hence
$$\End_{\sM_{hom}}(\bar h^{tr}_4((X)) \simeq  \End_{\sM_{hom}}(\bar h(X)_{prim}) \simeq \Q[id]$$
 If $h(X)$ is finite dimensional then the indecomposability of $\End_{\sM_{hom}}(\bar h^{tr}_4(X))$ in $\sM_{hom}(\C)$ implies the indecomposability in $\sM_{rat}(\C)$. Therefore  
 $$\End_{\sM_{rat}}(t(X)) \simeq \End_{\sM_{rat}}(h(X)_{prim}) \simeq \Q[id]$$
and the transcendental motive of $X$ is  indecomposable.\par
\end{proof}   

\begin{rk}\label{congetture} Let $X$ be a cubic fourfold and $l$ a general line in $X$. By Prop. \ref{summand} there is an isomorphism of motives

$$\Pr(S_l/Y_l) \simeq t_2(S_l)^{-}(1) \simeq t(X)$$

Suppose that the Prym  motive is isomorphic to the  (twisted) transcendental  motive of a smooth surface $Z$. Then, by  Prop. \ref{superfici} (i),  $X$ is special, i.e. $\dim A^2(X) \ge 2$. By the same argument as in the proof of Prop. \ref{superfici} (i) the surface $Z$ has geometric genus $p_g(Z)=1$ and there exists a K3 surface $S$ such that  

$$H^2_{tr}(Z,\Q) \simeq H^2_{tr}(S,\Q).$$

By assuming that the Hodge conjecture holds for $Z \times S$, the above isomorphism is induced by a correspondence $\Gamma \in A^2(Z \times S)$, see \cite{Mo}.
Therefore $\Gamma$ gives a map of transcendental motives  $t_2(Z) \to t_2(S)$, that induces an isomorphism on the transcendental cohomology and hence, assuming Kimura's conjecture, is an isomorphism in $\sM_{rat}(C)$. Then $t(X) \simeq t_2(S)(1)$. \end{rk}

 \section{Rationality conjectures}

Let $X$ be a cubic fourfold. It was conjectured in \cite{Kuz} that $X$ is rational if and only if there exists a semi-orthogonal decomposition of the derived category $\mathbf{D}^b(X)$ of bounded complexes of coherent sheaves

$$\mathbf{D}^b(X) =<\sA_X,\sO_X,\sO_X(1),\sO_X(2)>, $$

\noindent such that $\sA_X$ is equivalent to the category $\mathbf{D}^b(S)$ where $S$ is a K3 surface.  If $X$ has an associated K3 surface $S$ , in the sense of Kuznetsov, then the motive $h(S)$ is uniquely determined, up to isomorphisms, by $X$. Let $\mathbf{D}^b(S_1) $ and $\mathbf{D}^b(S_2)$ be equivalent. It was conjectured by Orlov that this implies that the motives $h(S_1)$ and $h(S_2)$ are isomorphic. The conjecture has been proved in \cite{DelP-P} in the case $h(S_1)$ (and hence also $h(S_2)$) is finite dimensional and recently extended by D.Huybrechts  in \cite{Huy 2}  to all K3 surfaces over an algebraically closed field. \par
\noindent Let us denote by $\sC$ the moduli space of smooth cubic fourfolds. As it is customary, we will denote by $\sC_d\subset\sC$ the irreducible divisors that parametrize special cubic fourfolds with an intersection lattice whose determinant is $d$. Let $X$ be a general cubic fourfold inside $\sC_d$, where $d$ satisfies the following condition:

\medskip

(**)  $d$ is not divisible by 4,9 or a prime $p \equiv 2 (3)$.

\medskip

\noindent Hassett \cite{Has 1} has shown that $X \in \sC_d$ has an associated K3 surface, in the sense of Def. \ref{associated}, if and only if  satisfies (**). Then Addington and Thomas in \cite{AT} proved that a general such $X$ has an associated K3 surface in the sense of Kuznetsov.  Therefore, for a general cubic fourfold, Kuznetsov conjecture is equivalent  to the following conjecture, that has been certainly around for a while. 

\begin{conj}\label{motivkuz} A cubic fourfold $X\subset \P^5$ is rational if and only if it is contained in $\sC_d$, with $d$ satisifying (**).
\end{conj}

\begin{prop}\label{associatedmotive} Let $X$ be a  cubic fourfold in $\sC_d$,  where $d$ satisfies (**). Assuming  Kimura's conjecture, there exists a K3 surface $S$ and an isomorphism of motives $t_2(S)(1) \simeq t(X)$. \end{prop}

 \begin {proof} 
By Theorem 1.2 in \cite{AT} there exists a polarized K3 surface $S$ of degree $d$ and  a correspondence $\Gamma \in A^3(S \times X)$ which induces an Hodge isometry between the (shifted) primitive cohomology of   $S$ and   the lattice $<\gamma^2,Z>^{\perp} $ inside $H^4(X,\Z)$. Here  the class of $Z$ is not homologous to $\gamma^2$. Let  $PHS_{\Q}$ be the semisimple abelian category  of polarized Hodge structures. Then  $\Gamma$ induces an isomorphism between the polarized Hodge structures $T(S)_{\Q}(1)$ and $T(X)_{\Q}$ in $PHS_{\Q}$,  where $T(S)$ and $T(X)$ are the transcendental lattices of $S$ and $X$ respectively. Let $\sM^B_{hom}(\C)$ be the subcategory of $\sM_{hom}(\C)$ generated by the homological motives $h_{hom}(X)$ of smooth complex projective varieties $X$ satisfying the standard  conjecture $B(X)$. Since $B(X)$ implies the standard conjecture $D(X)$, for smooth  varieties over $\C$, the category $\sM^B_{hom}$ is contained in the  category $\sM_{num}(\C)$ of numerical motives and hence it is semisimple. The  Hodge realization  functor 

$$ H_{Hodge} : \sM_{rat} (\C)\to PHS_{\Q},$$ 

\noindent factors trough $\sM^B_{hom}(\C)$ and the induced functor $H_{Hodge} : \sM^B_{hom}(\sC) \to PHS_{\Q}$ is faithful and exact. Both the K3 surface $S$ and the cubic fourfold $X$ satisfy $B(X)$ and hence $M_{hom}$ and $N_{hom}$ belong to $\sM^B_{hom}(\C)$, where $M_{hom}$ and $N_{hom}$ are the  images of $t_2(S)(1)$ and $t(X)$ in $ \sM_{hom}(\C)$, respectively. Then $M_{hom}$ and $N_{hom}$ have isomorphic  images in $PHS_{\Q}$ and hence the correspondence $\Gamma$ induces an isomorphism between $M_{hom} $ and $N_{hom}$ in $\sM^B_{hom}(\C)$. By Kimura's conjecture on the finite dimensionality of motives   the functor $F : \sM_{rat}(\C)\to \sM^B_{hom}(\C)$ is conservative, i.e. it preserves isomorphisms, see \cite[Thm. 8.2.4]{AK}. Therefore  the correspondence $\Gamma$ gives  an isomorphism between $t_2(S)(1)$ and $t(X)$ in $\sM_{rat}(\C)$. \end{proof}
Conjecture \ref{motivkuz} and Prop. \ref{associatedmotive} clearly suggest the following

\begin{conj}\label{raziomotiva}
If a cubic fourfold $X$ is rational then there exist   a K3 surface $S$ and an isomorphism  of transcendental motives $t(X) \simeq t_2(S)(1)$.
\end{conj}

  \section {Cubic fourfolds  fibered over a plane}
  
Let $X$  be a rational cubic fourfold Then by considering the surfaces blown up in a birational map  $\rho : \P^4 \dashrightarrow X$ one sees that  there are smooth projective surfaces $S_1,\dots S_n$ such that the transcendental  motive $t(X)$ is a direct summand of 
$\sum_{1 \le i \le n}t_2(S_i)(1)$, see \cite[Prop. 17]{Has 2}. Therefore the motive $h(X)$ is finite dimensional if all the surfaces $S_i$ have finite dimensional motives. 
Y.Zarhin  in \cite{Za} gives a restriction  for  the types  of surfaces that can appear when  resolving the indeterminacy  of  $\rho$. 

\medskip

\noindent Let us check two examples where we have a K3 surface $S$, and  $t(X) \simeq t_2(S)(1)$. This is the case for instance if $X$ contains two planes. Then $X$ is rational and $S$ is a K3 surface, a complete intersection of hypersurfaces of bidegrees (1,2), that is the indeterminacy locus of a birational map $\rho: \P^2 \times \P^2 \dashrightarrow X$. The map $\rho$ is  defined by taking the unique line trough a point $P$ joining the two planes and intersecting with X, see \cite[1.2]{Has 2}. If  $X$ is a generic cubic fourfold in $\sC_{14}$ then $X$ is rational and there is a birational map $\rho : Q \to X$, where $Q$ is a smooth quadric hypersurface in $\P^5$.  The indeterminacy locus of $\rho$ is a surface  $S'$ birationally equivalent to a K3 surface $S$ such that $F(X) \simeq S^{[2]}$, see   \cite[Thm. 2.2]{BRS}). Therefore, by Thm. \ref{mapiso}, we get $t(X) \simeq t_2(S)(1) =t_2(S')(1)$.

\medskip

In this section we consider the case of a cubic fourfold $X$ endowed with a rational map $f: X\dashrightarrow \P^2$ such that the fibration $f : \tilde X \to \P^2$ obtained by resolving the base locus of $f$ has rational fibers. Then, according to \cite[(2)]{Vial 1} the motive $h(\tilde X)$ splits as follows 

\begin{equation} \label{decoVial}
h( \tilde X) \simeq h(\P^2) \oplus h(\P^2)(1) \oplus h(\P^2)(2) \oplus M(1)
\end{equation}

where M is isomorphic to a direct summand of the motive of some smooth surface $Z$. Suppose that $f$ is obtained via a linear system having a smooth surface $T \subset X$ such that $t_2(T)=0$ as base locus, and call $ \tilde X \to X$ the blow-up along $T$. Then $h(\tilde X) \simeq h(X) \oplus h(T)(1)$ and  

\begin{equation}\label{mappahom} 
A_1( \tilde X)_{hom} \simeq A_1(X)_{hom} \oplus A_0(T)_0, 
\end{equation}

where $A^3(\tilde X) =A^3(t(\tilde X))=A_1( \tilde X)_{hom}$ and $A^3( X) =A^3(t( X)) =A_1( X)_{hom}$. Since $t_2(T)=0$, by taking a reduced Chow-k\"unneth decomposition $h(T) =\un \oplus h_1(T) \oplus h^{alg}_2(T) \oplus h_3(T)\oplus \L^2$, we get

$$ A^3(h_3(T)(1))=A^2(h(T))= A^2(h_3(T) )=A_0(T)_0 \simeq (\Alb T)_{\Q}.$$
 
Therefore the isomorphism in  (\ref{mappahom}) implies $A^3(t(\tilde X)) \simeq A^3(t(X))\oplus A^3(h_3(T)(1))$. Since the other Chow groups vanish on both sides we get an isomorphism of motives  $t(\tilde X) \simeq  t(X) \oplus h_3(T)(1)$. The motive $h_3(T)$, being  a direct summand of the motive of $\Alb T$, is finite dimensional, see \cite[6.2.12]{MNP}. Therefore the motive $h(X)$ is finite dimensional if and only if $h(\tilde X)$ is finite dimensional that is the case if the surface $Z$ appearing in (\ref{decoVial}) has a finite dimensional motive.\par
\noindent Examples  of this situation are general cubic fourfolds $X$ belonging either to $\sC_8$ or to  $\sC_{18}$, that are conjecturally not rational. In the first case $T$ is a plane and the fibers  of $\pi$ are quadrics, in the second case $T$ is ruled elliptic and  the  fibers are  del Pezzo surfaces of degree 6.\par
\noindent In order to identify the surfaces $Z$ we will use the following  proposition, that comes from the results in \cite[Prop. 6.7]{Vial 1}.\
 
\begin{prop}\label{vialismo}
Let $X$ be a cubic fourfold containing a surface $T$  with $t_2(T)=0$ and let $\pi: \tilde X \to X$ be the blow-up of $X$ along $T$ with $E$ exceptional divisor. Let 
 $f : \tilde X \to \P^2$ be a surjective morphism. Let $D$ be the discriminant  curve of the fibration $f$ and let $B^o=\P^2- D$. Assume that for all $t \in B^o$, the  fibers $\tilde X_t$ are smooth rational surfaces. Then  there is a finite number of smooth surfaces $\tilde B_i$, for $i =1\dots n$, with surjective and  finite maps  $r_i  : \tilde B_i \to \P^2$, such that the motive $h(X)$ is finite dimensional if all the motives $h(B_i)$ are finite dimensional. In this case the transcendental motive $t(X)$ is a direct summand of $ \bigoplus_{1\le i \le n}t_2(B_i)(1)$

\end{prop}
\begin {proof} Let $t \in\P^2 $ and let $ j_t: \tilde X_t \to \tilde X$ be the inclusion. The induced map on  Chow groups $(j_t)_* : A_1(\tilde X_t) \to A_1(\tilde X)$ fits into the following diagram

$$ \CD  A_1(\tilde X_t)@>{(j_t)_*}>> A_1( \tilde X)\\
@VVV                      @V{cl}VV   \\
H^2(\tilde X_t)@>{(j_t)_*}>>   H^6(\tilde X,\Q) \endCD $$

Here $H^6 (X,\Q) \simeq \Q[\gamma^3/3]$, with $\gamma \in A^1(X)$ a hyperplane section. From the long exact sequence of cohomology groups

$$ \cdots \to H^n(X,\Q)) \to H^n(T,\Q)) \oplus H^n(\tilde X)\to H^n(E,\Q)\to H^{n+1}(X,\Q)\to \cdots$$ 
where $E$ is the exceptional divisor of the blow-up, we get an isomorphism  $H^6(\tilde X ,\Q) \simeq H^6(X, \Q)\oplus \Q$, with $\Q  \simeq H^6(E,\Q)$. Therefore the image of $ (j_t)_* $  lies in $A_1(\tilde X)_{hom}$. \par
\noindent From the isomorphism in (\ref{mappahom}) and the surjective homomorphism
 $$ A_1(X)_{hom}\oplus A_1(E)_{hom} \to A_1(\tilde X)_{hom}\to 0$$
 we get that the image of $A_1(E)_{hom} $ in $A_1( \tilde X)$ is isomorphic to $A_0(T)_0$.  
Therefore the map

 $$\CD  \bigoplus_{t \in \P^2} A_1(\tilde X_t)@>{(j_t)_*} >>A_1(\tilde X)_{hom}, \endCD $$
 
when composed with the projection $A_1(\tilde X)_{hom} \to A_1(X)_{hom}$, gives  a surjective map

\begin{equation}\label{mappafibre}
 \bigoplus_{t \in \P^2} A_1(\tilde X_t) \to A_1(X)_{hom}.
\end{equation}

\noindent  Let $\sH=\Hilb_1(\tilde X/\P^2)$ be the relative Hilbert scheme whose fibers  parametrize curves in the fibers of $f$. Let
 
  $$\CD  \sC @>{q}>>  \tilde X \\
 @V{p}VV            @.  \\
\sH    \\
 @V{\pi}VV \\
 \P^2  \endCD $$
 
be the incidence diagram,  where $\sC$ is the universal family over  $\sH$, i.e. $\sC =\{ (C,x) | x \in C\} \subset \sH \times \tilde X$. Then the map

$$p^*q_* : A_0(\sH) \to A_1( \tilde X)_{hom} \to A_1(X)_{hom}$$

factors trough $A_0(\sH) \to A_1(\tilde X_t)$ and $f_t : A_1(\tilde X_t) \to A_1(X)_{hom}$, for every fiber $\tilde{X}_t$.  By \cite[Lemma 6.6]{Vial 1} there is  finite  set  $\sE =\{\sH_1,\dots,\sH_n\}$ of irreducible components  of $\Hilb_1(\tilde X/\P^2)$, such that they obey the following technical condition: 
 
 \medskip
 
\it
$\forall t \in  B^o$, the set $\{cl(q_*[p^{-1}(u)]/  u \in \sH_i, t=\pi(u)\}$ span 
$H^2(\tilde X_t,\Q).\ \ \ \ \ (*)$ 
\rm.
\medskip

Let $f_i : \tilde \sH_i \to \sH_i$ be a resolution of singularities. By \cite[Prop. 6.7]{Vial 1}, for all $i$ there are smooth linear sections $\tilde B_i \to \tilde \sH_i$ of dimension 2, such that, for every $i \in (1,\dots,n)$ the  following composed map $r_i: \tilde B_i \to \P^2$ is surjective:

\begin {equation}\label{surfaces}
\CD r_i : \tilde B_i \to \tilde \sH_i@>{f_i}>> \sH_i \to \P^2 \endCD 
\end{equation}

The map $r_i$ is finite and, for any $t \in \P^2$, $r_i^{-1}(t) $ contains  a point in every connected component of the fiber of $\sH_i$ over $t$.  Let  $\tilde B = \coprod_{1 \le i \le n} \tilde B_i$  be the disjoint union of the surfaces $\tilde B_i$. Again by \cite[Prop. 6.7]{Vial 1}, there is a correspondence $\Gamma \in A^3(\tilde B  \times \tilde X)$ such that  $\Gamma = \oplus_i \Gamma_i $ where $\Gamma_i \in A^3 (\tilde B_i \times \tilde X)$ is the class of the image of $\sC_i$ inside $\tilde B_i \times \tilde X $ in the incidence diagram

 \begin{equation}\label{diagramunicurve}
  \CD  \sC_i @>{q_i}>> \tilde X \\
 @V{p_i}VV            @.  \\
\tilde B_i
 \endCD 
 \end{equation}

Here  $p_i : \sC_i \to \tilde B_i$ is the pullback of the universal family $\sC \to \sH_i$ along $\tilde B_i \to \tilde \sH_i \to \sH_i$.The correspondence $\Gamma \in A^3(\tilde B \times X)$ gives a map of motives

\begin{equation}\label{mapmotives} (\Gamma)_*: h(\tilde B) (1)= \sum_{1 \le i \le n} h(\tilde B_i)(1) \to h(X).\end{equation}

Let $h_i : \tilde R_i \to \tilde R_i$ be the normalization of the curve $R_i =r^{-1}_i(D)$, with $r_i$ as in (\ref{surfaces}) and $D$ the discriminant curve. For  every  $P \in D \subset \P^2$  and $i \in \{1,\dots n\}$, $r_i^{-1}(P)$ is a finite set of points, each one in a connected component of  the fiber $\tilde \sH_{i,P}$. Let us set

$$h: \tilde R := \coprod_{1 \le i \le n} \tilde R_i \to \tilde B$$

and  $\Gamma_{\tilde R} :=h^*(\Gamma) \in A^3(\tilde R \times \tilde X)=A_2( \tilde R \times \tilde X)$.
Then $\Gamma_{\tilde R}$  gives a map of motives $h(\tilde R)(1)\to h(X)$. Let $h( \tilde R) =\un \oplus h_1(\tilde R)\oplus \L$ be a C-K decompositon and let $h_1(\tilde R)(1) \to t(X)$ be the map induced by $\Gamma_{\tilde R}$. Then the associated maps on  Chow groups vanish, because  $A^p(h_1(\tilde R) (1))=A^{p-1}(h_1( \tilde R)) \ne 0$ only for $p=2$ while $A^2(t(X)=0$. In particular the composite map 
\begin{equation}\label{compmap}\CD A_0(\tilde R)_0= \bigoplus_{1\le i\ \le n} A_0( \tilde R_i)_)@>{h_*}>>A_0(\tilde B)_0@>{\Gamma_*}>>A_1(X)_{hom}\endCD
\end{equation} 

is 0. Let 

$$(j_D)_*: \bigoplus_{P\in D} A_1(\tilde X_P) \to A_1(X)_{hom}$$

be the sum of the $(j_P)_*$ and let us define $g_D : \bigoplus_{P\in D} A_1(\tilde X_P) \to A_0(\tilde R)_0$ by sending a class $[C] \in A_1(\tilde X_P)$, lying on a connected component of the fiber $\sH_{i,P}$, to  the class of the corresponding point $x_C \in r^{-1}(P)$ in $A_0( \tilde R)$. Then $(j_D)_*$ factors trough $g_D$ and the map  $A_0(\tilde R)_0 \to A_1(X)_{hom}$ in (\ref {compmap}). Therefore the map $(j_D)_*$ vanishes. 

\medskip

\noindent From \ref{mappafibre} we get

$$A_1(X)_{hom} = \Im(\bigoplus_{t\in \P^2} A_1(\tilde X_t) \to A_1( X)_{hom})= \Im(\bigoplus_{t\in B^o} A_1(\tilde X_t) \to A_1( X)_{hom}) \oplus$$
$$\oplus  \Im(\bigoplus_{P \in D}  A_1(\tilde X_P) \to A_1( X)_{hom});$$

\noindent where the map$\bigoplus_{P \in D}  A_1(\tilde X_P) \to A_1( X)_{hom}$ is 0. Therefore

from \cite[Prop. 6.7]{Vial 1} we get

$$A_1(X)_{hom}=\Im (\bigoplus_{t\in B^o} A_1(\tilde X_t) \to A_1( X)_{hom}) \subseteq \Im( \Gamma_* : A_0(\tilde B)_0 \to A_1(X)_{hom}).$$

\noindent and thus 

$$A_1(X)_{hom} =  \Im( \Gamma_* : A_0(\tilde B)_0 \to A_1( X)_{hom}) $$ 

Therefore the map of motives in (\ref{mapmotives}) induces a map on Chow groups 

\begin{equation}\label{chowgroups}A^3(h(\tilde B)(1))  = A_0(\tilde B)_0 \to A^3(t(X))=A_1(X)_{hom} \end{equation}
that is surjective. Let

$$h(\tilde B_i) =\un \oplus   h_1(\tilde B_i) \oplus  (\L)^{\rho_i} \oplus h^{alg}_2\oplus t_2(\tilde B_i)\oplus h_3(\tilde B_i) \oplus \L^2$$

be a reduced C-K decomposition and let $h(\tilde B) =\sum_{1 \le i \le n}h(\tilde B_i)$ be the corresponding decomposition for $\tilde B$.  Then

$$A^3((\tilde B)(1)) = A^3(t_2( \tilde B)(1)) \oplus A^3(h_3(\tilde B)(1)=  A^2(t_2(\tilde B) \oplus A^2( h_3(\tilde B)$$

Since the Chow groups  $A^i(t_2(\tilde B)(1)) \oplus  A^i(h_3( \tilde B)(1))$ and $A^i(t(X))$  vanish for $i \ne 3$, the transcendental motive $t(X)$ is a direct summand of 
$$t_2( \tilde B)(1)\oplus h_3( \tilde B)(1)=\sum_{1\le i \le n}((t_2(\tilde B_i)(1) \oplus h_3( \tilde B_i)(1)).$$
The motives $h_3( \tilde B_i)$, being of abelian type are finite dimensional. Therefore $t(X)$ is finite dimensional if  $t_2( \tilde B_i)$ is finite dimensional for every $i \in \{1,\dots n\}$. Assume that $t_2(\tilde B_i)$ is finite dimensional : than  $t(X)$ cannot be isomorphic to a direct summand of $M=h_3(\tilde B)(1)$, hence it is a direct summand of $t_2(\tilde B)(1)$. Suppose, on the contrary, that   $M \simeq N \oplus t(X)$ and let $f : M  \to t(X)$ and $g : t(X) \to M$ such that $f \circ g= id_{t(X)} = \pi^{tr}_4$, where $t(X)=(X,\pi^{tr}_4)$. Since $M$ and $t(X)$ are finite dimensional of different parity the map $f$ is smash-nilpotent, see [MNP, Prop.5.3.1].  Smash-nilpotent correspondences form a bilateral ideal  $\sI$  in the category $\sM_{rat}$, see \cite[7.4.3]{AK}, and  hence we get 
$id_{t(X) }\circ f \circ g=id_{t(X)} \in \sI$. Therefore, the projector $\pi^{tr}_4 =id_{t(X)}$, being smash-nilpotent, vanishes and we get a contradiction: $t(X)=0$.

 \par
 \end{proof}
 
 \subsection{Cubic fourfolds containing a plane}

 Let $X$ be a generic fourfold in $\sC_8$. Then $X$ contains a  plane $P$. Call $\tilde{X}$ the blow-up of $X$ along $P$ and $\pi:\tilde{X}\to \P^2$ the morphism that resolves the projection off $P$. The morphism $\pi$ is a fibration in quadric surfaces, whose fibers degenerate along a plane sextic $C$, which is smooth in the general case. The double cover $r : S \to \P^2$ ramified along $C$ is a K3 surface. Recall that the relative Hilbert scheme of lines $\mathcal H(0,1)$ of the morphism $\pi$ is an \'{e}tale projective bundle over $S$. The Stein factorization of the map $\sH(0,1) \to \P^2$ yields  $\sH(0,1)\to S \to \P^2$, where the first map is a $\P^1$-bundle.
 \begin{prop}\label{plane}
If $X$ is a general element in $\sC_8$  the transcendental motive $t(X)$ is isomorphic to  the motive $t_2(S)(1)$. Therefore if the motive of $S$ is finite dimensional then also $h(X)$ is finite dimensional .\end{prop}

\begin{proof} Since $A_0(\P^2)_0=0$ we have $A_1(X)_{hom} =A_1( \tilde X)_{hom}$. In order to show that $\sH_1= \sH(0,1)$ is the only component that we need to apply Prop. \ref{vialismo}, we need to check that the technical condition $(*)$ holds true for this Hilbert scheme. This is not hard to show, since the $H^2(\tilde X_t, \mathbf Q)$, for $t \in \P^2 -C$, is generated by the classes of any line of the two rulings of the quadric. In fact the two irreducible components $ \sH^{(1)}_1$ and $\sH^{(2)}_1$ of $\sH(0,1)$ over the point $t \in \P^2$, not lying on the discriminant, parametrize the lines in each ruling. From (\ref{mapmotives}) we get a map 
$t_2(S)(1) \to t(X)$ such  that $A^3(t_2(S)(1))= A_0(S)_0 \to A^3(t(X) =A_1(X)_{hom}$ is surjective  and we are left to show that it is an isomorphism. By \cite[Theorem 3.6]{SYZ} there is an isomorphism  $A_0(S)_0 \to A_0(F)_2$, where $F =F(X)$, that, together with the isomorphism $A_0(F)_2\simeq A_1(X)_{hom}$ in Theorem \ref{fano},  gives the isomorphism $A_0(S)_0 \simeq A^3(t(X))=A_1(X)_{hom}$.
\end{proof}  
 
 \begin{rk}  The above example  suggests that the statement in Conj. \ref{raziomotiva} cannot be inverted, in the sense that a generic element of $\sC_8$ is conjecturally not rational and yet there is an isomorphism $t(X) \simeq t_2(S)(1)$, with $S$ a K3 surface.  A similar result appears in \cite[Thm. 0.3]{Bull}, for all $X \in \sC_d$, where $d$ satisfies the following numerical  condition : $d =k^2 d_0$, with $k\in \Z$ and  $d \vert 2n^2+2n+2$, with $n \in \Z$.
 \end{rk} 
 
  \subsection{Cubic fourfolds fibered in del Pezzo sextics}

 Let  $X$ be a generic  fourfold in $\sC_{18}$. The fourfold $X$ contains an elliptic ruled surface $ T$ of degree 6 such that the linear system of quadrics in $\P^5$ containing $T$ is two dimensional. Let once again $r : \tilde X\to X $ be the blow-up of $X$ at $T$ and  $\pi : \tilde X \to \P^2$ the (resolution of the) map induced by the linear system of quadrics containing $T$. The generic fiber of $\pi$ is a del Pezzo surface of degree 6. The generic del Pezzo fibration $\pi$  obtained from a cubic fourfold in $\sC_{18}$ is a {\it good del Pezzo fibration} in the sense of \cite[Def. 11]{AHTV-A}. The discriminant curve  $D$ of $\pi$ has  two irreducible components, a smooth sextic $C$ and a sextic $\bar C$ with 9 cusps. As in the previous case the double cover $S \to \P^2$ branched on $C$ is a smooth K3 surface of degree 2. The goal of this section is to show that there is an isomorphism $t_2(S)(1) \simeq t(X)$. The main difference with the $\mathcal{C}_8$ case is that here the Picard rank of the generic fiber is higher, so we will need to consider surfaces inside two different Hilbert schemes of curves in order to obey the technical condition $(*)$ and hence to apply the constructions of Prop. \ref{vialismo}.
 
\medskip
 
Associated to the good del Pezzo fibration $\pi : \tilde X \to \P^2$  there is  a non-singular degree 3 cover $f : Z \to \P^2$ branched along a cuspidal sextic $\bar{C}$ (see \cite{AHTV-A}) where $Z$ is a non singular surface. Let $\sH(0,2)\to \P^2$ be the relative Hilbert scheme  of connected genus 0 curves of anti canonical degree 2 on the fibers. The Stein factorization yields an \'{e}tale $\P^1$-bundle $\pi_1: \sH(0,2) \to Z$. It is easy to see that, on every fiber, the $\P^1$-bundle is given by the strict transform of the lines through each of the 3 blown-up points $P_1,P_2,P_3\in \P^2$ of the corresponding del Pezzo of degree 6. This gives a diagram  
 $$\CD  \mathcal H(0,2) \\
 @V{\P^1}VV            @.  \\
 Z \\
 @V{f_1}VV \\
 \P^2  \endCD .$$
 
 \begin{prop}
The triple cover $ Z  \to \P^2 $ is an elliptic ruled surface and hence $t_2(Z) =0$ and $A_0(Z)_0 \simeq (\Alb Z)_{\Q} \simeq (\Jac E)_{\Q}$, with $E$ an elliptic curve.
\end{prop}

\begin{proof}Let $\bar{C} \subset \P^2$ be the ramification locus of the triple cover $f_1:  Z\to \P^2$. As it has been observed in \cite{AHTV-A}, for a generic cubic $X\in\sC_{18}$, $\bar{C}$ is a cuspidal degree 6 curve with 9 cusps. It is well known \cite{Mir} that such a triple cover is completely determined by the Tschirnhausen rank two vector bundle on $\P^2$ and a section of (a twist of) the relative $\mathcal{O}(3)$ on the associated projectivized $\P^1$-bundle. Let us denote $V$ the Tschirnhausen module. From Prop. 4.7 of \cite{Mir} we see that $\bar{C}$ belongs to the linear system $|-2c_1(V)|$, hence $c_1(V)=\sO_{\P^2}(-3)$. Then, by \cite[Lemma 10.1]{Mir},  the number of cusps is exactly $3c_2$, this means that $c_2(V)=3.$ With these data in mind we can use \cite[Prop. 10.3]{Mir} to compute the invariants of $Z$ and get

$$\chi=0,\ \ K^2=0,\ \  e(Z)=0.$$

\noindent Now, by \cite[Cor 2.3]{Shi} we see that that $V \cong \Omega_{\mathbb{P}^2}$, hence by \cite[Cor 10.6]{Mir} we have $p_g(Z)=0$, and $q(Z)=1$. This easily implies that the surface $Z$ is again an elliptic ruled surface. Note that such a triple plane being  an  elliptic ruled surface was first observed by Du Val in \cite{DV} by different methods. Since $p_g(Z)=0$ and  $Z$ is not of general type we get $t_2(Z)=0$. The rest follows from the isomorphism $A_0(Z)_0 \simeq (\Jac E)_{\Q}$.

\end{proof}
 
As we have already anticipated, in this case, considering just one Hilbert scheme will not be enough in order to apply Prop. \ref{vialismo}, since the fibers of $\pi$ have higher Picard rank. Hence we need to consider also $\mathcal H(0,3)$, the relative Hilbert scheme of curves of genus zero and canonical degree 3 inside the fibers. There are two 2-dimensional families of such curves on a del Pezzo sextic. One is given by the strict transforms of the lines in $\P^2$ that do not pass through any of the three base points. The second is given by conics passing through the three base points. We will call the former cubic curves of first type and the latter cubic curves of second type. The Stein factorization  $ \sH(0,3)\to  S\to \P^2  $  of the natural projection $\pi_2 :\mathcal{H}(0,3)\to \P^2$ 
 reflects this difference and displays $\mathcal{H}(0,3)$ as an \'{e}tale $\P^2$-bundle over a smooth degree two K3 surface $S$ \cite{AHTV-A}. It is straightforward to see that one $\P^2$ parametrizes the curves of first type and the other those of second type.\par
\noindent In order to apply Prop. \ref{vialismo} to the fibration $\tilde X \to \P^2$, we  prove  the following lemma.

\begin{lm}\label{teccond}
Let $\pi : \tilde X \to \P^2$ a del Pezzo fibration and let $D \subset \P^2$  be the discriminant curve. Then  the two components $\sH(0,2)=\sH_1$ and $\sH(0,3)=\sH_2$ of the Hilbert scheme $\sH/\P^2$ obey the technical condition $(*)$, i.e.
  $\forall t \in  B^o=\P^2 -D $, the set $\{cl(q_*[p^{-1}(u)]/  u \in \sH_i, t=\pi(u)\}$ 
  span $H^2(\tilde X_t,\Q)$, where $\pi_i : \sH_i \to \P^2$. 
 \end{lm}

\begin{proof}
 Fix a point $p\in \P^2$, such that the fiber $\tilde{X}_p$ over $p$ is a smooth del Pezzo sextic. Its Picard rank is 4 and the generators are the proper transform of a line and the three exceptional divisors. Let us denote $H,\ E_1, E_2$ and $E_3$ these divisor classes. Then, the fiber $(\sH_2)_p \subset \mathcal{H}(0,3)$ over $p$ contains at least a curve from the linear system $|H|$ and a curve from the linear system $|2H- E_1-E_2-E_3|$. On the other hand, the fiber $(\sH_1)_p \subset \mathcal{H}(0,2)$ over $p$ contain at least 3 curves from the linear systems $|H-E_1|$, $|H-E_2|$ and $|H-E_3|$. It is straightforward to see that linear combinations of these 5 divisor classes generate the whole $H^2(\tilde{X}_p,\mathbf{Q})$.\par
\end{proof}

\begin{thm}\label{del Pezzo fibr.} 
Let  $X$ be a generic  fourfold in $\sC_{18}$. Then $t(X)$ is isomorphic to $ t_2(S)(1)$.
\end{thm}

\begin{proof} From Lemma \ref{teccond} it follows that we can apply Prop. \ref{vialismo} to the Hilbert schemes $\sH_1$ and $\sH_2$. Hence, as in (\ref{surfaces}), we get smooth surfaces  $\tilde B_1 $ and $\tilde B_2$ with finite maps $r_1 : \tilde B_1 \to \P^2$ and $ r_2: \tilde B_2 \to \P^2$ that induce isomorphisms $A_0(\tilde B_1)_0=A_0(Z)_0$ and $A_0( \tilde B_2)_0 =A_0(S)_0$. The discriminant curve $D \subset  \P^2 $ has two irreducible components, the cuspidal sextic $ \bar C$  and the smooth sextic $C$.  Let $\tilde R_1\to  R_1$ and $\tilde R_2 \to R_2$ be  respectively the normalization of $R_1= r_1^{-1}(\bar C)$ and the normalization of $R_2=r_2^{-1}(C)$.  The surfaces $ Z$ and $S$ have reduced Chow-K\"unneth decompositions with  $t_2(Z)=0$ and   $h_3( S) =0$.
The curve $\tilde R_1$ is smooth of genus 1 and hence it is an elliptic curve birational (and hence isomorphic) to a curve $E$, such that $Z$ is birational to the product $\P^1 \times E$. Therefore
 $$A_0(\tilde R_1)_0 \simeq  A_0(Z)_0 =(\Jac E)_{\Q}.$$
 The K3 surface $S$ is a double cover of $\P^2$ ramified along $C$. The curve $ \tilde R_2 $, is a constant cycle curve( see  \cite[7.1]{Huy 1}),
 hence the map $ A_0(\tilde R_2)_0 \to A_0( S)_0 $ is the 0-map. From (\ref{compmap}) we get  that the map

$$A_0(\tilde R_1)_0 \oplus A_0(\tilde R_2)_0 \to   A_0(\tilde B)_0\simeq (A_0(Z)_0  \oplus A_0( S)_0) \to A_1(X)_{hom} $$

vanishes.Therefore   the map $A_0(\tilde R_1) =A_0(Z)_0 \to A_1(X)_{hom}$ is 0. The map $(\Gamma)_*)$ in (\ref{mapmotives}) gives a map

$$(\Gamma)_*: t_2(S)(1) \to t(X),$$
 such that the associated map on Chow groups $\A_0(S)_0 \to A_1(X)_{hom}$ in \ref{chowgroups} is surjective. In order to show that is also injective and hence  the map $(\Gamma)_*$  gives an isomorphism of motives, we apply  the same argument as in the proof of \cite[Thm. 3.6]{SYZ}. If  $\Psi(\alpha) = 0$ in $A_1( X)_{hom}$, with $\alpha \in A_0( S)_0$ then $ \sigma_*(\alpha) =0$, where $\sigma$ is the involution on $S$ coming from the double cover $ S \to \P^2$. Therefore $\alpha=0$. 
 \end{proof}  
 
 \begin{rk} G.Tabuada in \cite[Thm. 1.7]{Tab}  proves a result, similar to Thm. \ref{del Pezzo fibr.}, for a fibration $f: Y \to \P^2$, where $Y$ is a smooth 4-fold and the fibers of $f$ are sextic del Pezzo surfaces. There are finite flat morphisms $Z_2 \to\P^2$ , $Z_3 \to \P^3$ of degree 3 and 2 respectively,  such that, if the motives of $Z_2$ and $Z_3$  are {\it Schur-finite}, then  the motive of   $Y$ is Schur-finite. Note that, since a  finite dimensional motive is also Schur-finite, the motive of the surface $Z$ in(5.12) is Schur-finite 
 
\end{rk}
 
\section{Cubic fourfolds with an involution}

  Let $\sigma$ be  the  involution  on $\P^5$  defined by  
$$[x_0,x_1,x_2,x_3,x_4,x_5] \to [x_0,x_1,x_2,x_3,- x_4, - x_5] $$ 

A cubic fourfold $X$ fixed by $\sigma$ has an equation of the form 

\begin{equation}\label{cubicfu} 
C(x_0,x_1,x_2,x_3) + x^2_4L_1 +x^2_5L_2 + x_4x_5L_3=0
\end{equation}

where  $C$ has degree 3 and  $L_1, L_2,L_3$ are linear forms in $x_0,x_1,x_2,x_3$. C. Camere shows that this is the unique automorphism of $\P^5$ inducing a symplectic involution on $F(X)$ \cite[Sect. 7]{Ca}. The  locus of fixed points of $\sigma$ on $\P^5$  is the disjoint union of a $\P^3$ defined by $x_4 =x_5 =0$ and the line  $r$ joining the base points $P_4$ and $P_5$. The line $r$  is contained in $X$ and the fixed locus on $\P^3$ is the cubic surface $C=0$. The symplectic involution  $\sigma_F$ on $F(X)$ has 28 isolated points, i.e.  the line $r$ and the 27 lines on the cubic surface, plus a K3 surface $S$, consisting of the  lines joining a fixed point $Q_1$ on $\P^3$ and a point $Q_2$ on $r$ (see again \cite{Ca}).
Let us now project with center the line $r$ and let $\tilde{X}$ denote the blow-up of $X$ along $r$. The projection resolves into a morphism $\delta: \tilde{X} \to \P^3$, which is well-known to be a conic bundle with quintic degeneration locus $D$. 

\begin{lm}
The quintic hypersurface $D\subset \P^3$ has a cubic and a quadric irreducible components. For appropriate choices of the $L_i$ and of $C$ the sextic intersection curve is smooth and parametrizes rank one conics. For general choices of the $L_i$ the quadric has rank 3.
\end{lm}

\begin{proof}

Let $p:=[a:b:c:d]\in\P^3$, in order to study the conic over $p$ we need to study the 
intersection of $X$ with the plane $\P^2_p:= \langle p,P_4,P_5\rangle \subset \P^5 $, where $\langle \cdot \rangle$ denotes as usual the linear span, and $P_4, P_5$ are the base points on $r$. Hence we substitute inside equation $\ref{cubicfu}$ the values $\lambda[0:0:0:0:1:0]+\mu[0:0:0:0:0:1]+\gamma[a:b:c:d:0:0]$, with $\lambda,\mu,\gamma \in \mathbb{C}.$ Recall that in this plane the equation of $r$ is $\gamma=0$. Then, dividing by $\gamma$ the cubic equation in $\gamma,\lambda$ and $\mu$ we obtain

\begin{equation}\label{conica}
\gamma^2C(a,b,c,d)+\lambda^2L_1(a,b,c,d)+\mu^2L_2(a,b,c,d)+\lambda \mu L_3(a,b,c,d).
\end{equation}

This is the conic obtained from the symmetric matrix

$$ \left(
  \begin{array}{ccc}
  C(a,b,c,d) & 0 & 0 \\
  0 & L_1 & \frac{1}{2}L_3 \\
  0 & \frac{1}{2}L_3 & L_2 
  \end{array}
\right).$$

Hence one easily sees that the equation of $D$ is $C\cdot(L_1L_2 - \frac{1}{4}L_3)$. The mere equation $L_1L_2 - \frac{1}{4}L_3$ shows that the quadric has at most rank 3 and that for general $L_i$ this is the case. Let us denote by $Q$ the quadric surface. Suppose now $L_3=x_0-x_1-x_2$, $L_1=(t-z)$, $L_2=(t+z)$ and $C$ is the Fermat cubic. Then the quadric has equation $-(x-y-z)^2 + t^2 -z^2$ and rank 3. A quick Macaulay2 \cite{Mac2} routine shows that the intersection with the Fermat cubic is a smooth sextic curve $Y$, and from the matrix representing the conic one sees that the sextic curves parametrizes conics of rank 1.
\end{proof}

The surface $S$ is a double cover of the cubic surface $C=0$ ramified along the degree 6 curve $Y$. It is straightforward to see that $S$ parametrizes irreducible (linear) components of degenerate conics, that are fixed by the involution. If one takes the double cover $W \stackrel{2:1}{\to}Q$, ramified along $Y$, this parametrizes the irreducible components of degenerate conics that are not  fixed by the involution (except for double lines, parametrized by $Y$). It is a classical construction that double covers $W$ of quadric cones, ramified along a smooth genus 4 sextic are del Pezzo surfaces of degree 1, and the double cover is induced by the linear system $|-2K_W|$, where $K_W$ is the canonical bundle. We observe that by Kodaira vanishing it is easy to see that $q(W)=0$. \medskip
 \noindent The  Abel-Jacobi map induces an isomorphism

$$H^{3,1}(X) \simeq H^{2,0}(F(X)) \simeq H^{2,0}(S)$$ 

and hence $H^2_{tr}(F,\Q) \simeq H^2_{tr}(S,\Q)$. By \cite[Thm. 3.1]{Lat 2} there is a correspondence $\Gamma \in A^3(S \times X)$
inducing a surjective homomorphism $ A_0(S)_0 \to A_1(X)_{hom}$.\par
\noindent Let  $h(S) =\un \oplus h^{alg}_2 \oplus t_2(S) \oplus \L^2$ be  a Chow-K\"unneth decomposition and let 

\begin{equation} \label {Latmap}\Gamma_* :  t_2(S)(1) \to t(X)\end{equation} 

be the map of motives induced by $\Gamma$. We have

$$A^3(t(X)) = A_1(X)_{hom}   \ ; \  A^3(t_2(S)(1))  =A^2(t_2(S)) =A^2(S)_0$$

and $A^i(t(X)) =A^i(t_2(S)(1))=0$ for $i \ne 3$. Therefore $\Gamma$ induces a surjective map on all Chow groups and hence $t(X)$ is a direct summand of $t_2(S)(1)$. 

The following result shows that $\Gamma_*$ is in fact an isomorphism.

\begin{prop}
The map of motives in \ref{Latmap} is an isomorphism.\end{prop}
  
\begin{proof} It is enough to show that the surjective map   $A_0(S)_0 \to A_1(X)_{hom}$ is an isomorphism. Let  $\alpha =[l_1] - [l_2] \in A_0(S)_0$ be such that $\Gamma_*(\alpha)=0 $. 
Then $[l_1]=[l_2]$ corresponds to a double line on a singular fiber of $\delta : \tilde X \to \P^3$. Let $\tau$ be the involution on $S$ associated to  the double cover $S \to C$, that is  ramified along the sextic $Y$. Since $Y$ parametrizes double lines on the singular fibers we have that $[l_1]=[l_2]$ belong to $Y$ (more precisely to the branch locus inside $S$, which is isomorphic to $Y$). The cubic surface $C$ is rational and hence $A_0(C)_0=0$. Therefore, by the same argument as in \cite[7.1]{Huy 1}, the fixed locus of $\tau$ is a constant cycle curve, i.e. the  map $A_0(Y)_0 \to A_0(S)_0$ vanishes, and we get $\alpha=0$ in $A_0(S)_0$.
\end{proof}


\begin{thebibliography}{10} 

\bibitem [Add]{Add}N.Addington, {\it On two rationality conjectures for cubic four folds}. Math. Res. Letters 23 (1), (2016) , 1-13. 

\bibitem [AT]{AT} N.Addington and R.Thomas, {\it Hodge theory and derived categories of cubic four folds}, Duke Math. J. {\bf 163}(2014), no.10,  1885-1927.

 \bibitem [AHTV-A] {AHTV-A}N.Addington,B.Hassett,Y.Tschinkel and A.Varilly-Alvarado,{\it Cubic fourfolds fibered in sextic  del Pezzo surfaces}, arXiv:1606.05321.

\bibitem [Am] {Am}, E.Amerik, {\it A computation of invariant of a rational self-map}, Ann. Fac. Sci.Toulouse Math., 18, (2009),no.3 ,445-457

\bibitem[An]{An} Y. Andr\'e, {\it Une introduction aux motifs,} Panoramas et
   synth\`eses, SMF, 2004.
   
 \bibitem [AK]{AK}  Y. Andr\'e and B.Kahn {\it Nilpotence, radicaux and structures monoidales},rend.Sem.Mat.Univ.Padova,Vol 108,(2002),107-291

 \bibitem [ABBV] {ABBV}A.Auel, M.Bernadara, M.Bolognesi, A.Varilly-Alvarado,{\it Cubic fourfolds containing a plane and a quintic del Pezzo surface}, Algebraic Geometry , Volume 1, Issue 2 (March 2014), p. 181-193

\bibitem [Ba]{Ba} K. Banerjee, {\it Algebraic cycles on the Fano variety of a cubic fourfold}, arxiv:1609.05627v1[math.AG], Sept 2016.

\bibitem[BD]{BD}A.Beauville and R.Donagi,{\it La variete' de droites d'une hypersurface cubique de dimension 4},C.R. Acad. Sci. Paris,Ser. I {\bf 301}91985)703-706.

\bibitem[BB]{BB} M.Bernardara and M.Bolognesi, {\it Categorical representability and intermediate Jacobians of Fano threefolds}, EMS Series of Congress Reports ''Derived categories in algebraic geometry'', p.1--25, EMS Ser. Congr. Rep., Eur. Math. Soc., Zurich, 2012.


\bibitem[BRS]{BRS}M.Bolognesi, F.Russo and G.Stagliano' {\it Some loci of rational cubic fourfolds}, arxiv:1504.05863,  Math. Ann. (2018). https://doi.org/10.1007/s00208-018-1707-7

 \bibitem [Bull]{Bull}T-H B\"ulles,{\it Motives of moduli spaces on K3 surfaces and of special cubic four folds},arXiv:1806.0828v1,(2018)
 
  \bibitem [Ca]{Ca} C.Camere ,{\it Symplectic involutions of holomorphic symplectic fourfolds}, Bull.Lond. Math. Soc. 44 no. 4(2012),687-702

 
 \bibitem[deC-M]{deC-M} A.de Cataldo and L.Migliorini, {\it The Chow Groups and the Motive of the Hilbert scheme of points on a surface}, Journal of Algebra {\bf 251}, (2002), 824-848.

 \bibitem [DelP-P] {DelP-P} A.Del Padrone and C.Pedrini {\it Derived categories of coherent sheaves and motives of K3 surfaces}, Regulators, Contemp. Math 571,AMS (2012),219-232
 
\bibitem [DV]{DV} P. du Val {\it On Triple Planes having Branch Curves of order not greater than twelve.},  Proc. London Math. Soc. S2-39 no. 1, 68.  
 
   \bibitem [GG]{GG} S.Gorchinskiy and V.Guletskii. {\it 
Motives and representability of algebraic cycles on threfolds over a field},
J.Alg. Geometry {\bf 21} (2012) 343-373.

   
 \bibitem [Has 1]{Has 1} B.Hassett, {\it Special cubic fourfolds}, Compositio Mathematica, {\bf 120} (2000)1-23

\bibitem [Has 2]{Has 2} B.Hassett, {\it Cubic fourfolds, K3 surfaces and rationality questions }, Lecture notes for a 2015 CIME-CIRM summer school.

   \bibitem [Huy 1]{Huy 1} D.Huybrechts (with an appendix by C. Voisin), { \it  Curves and cycles on K3 surfaces}. Algebraic Geometry 1 (2014), 69-106.

\bibitem [Huy 2 ] {Huy 2} D.Huybrechts, {\it Motives of derived equivalent K3 surfaces} arXiv:1702.03178 [math.AG] , February 2017.

  \bibitem[KMP]{KMP} B. Kahn, J. Murre and C. Pedrini, 

\newblock {\it On the transcendental part of the motive of a surface}, 
pp. 143--202 in "Algebraic cycles and Motives Vol II", 
\newblock London Math. Soc. LNS {\bf 344}, Cambridge University Press, 2008. 

\bibitem [Ki]{Ki} S.I.Kimura. {\it A note on 1-dimensional motives}  in "Algebraic cycles and Motives Vol II", 

\newblock London Math. Soc. LNS {\bf 344}, Cambridge University Press, 2008, 203-213.
  
  
  \bibitem [Kuz]{Kuz} A. Kuznetsov, {\it Derived categories and cubic fourfolds},Cohomological and geometric approaches to rationality problems, Progr.Math. vol 282, (2010)  219-243
 

\bibitem [Lat 1]{Lat 1} R.Laterveer. {\it A remark on the motive of the Fano variety of lines of a cubic}, arXiv:1611.08818v1 [math.AG] , November 2016

\bibitem [Lat 2] {Lat 2} R.Laterveer. {\it On the Chow group of certain cubic fourfolds}, arXiv:1703.03990v1 [math.AG] , March 2017


\bibitem[MNP]{MNP} J.Murre, J.Nagel and C.Peters, 
{\it Lectures on the theory of pure Motives}, 
AMS University Lectures Ser. Vol 61 (2013)

\bibitem[Mac2]{Mac2} D.R.Grayson and M.E.Stillman. {\it Macaulay2, a software system for research in algebraic geometry}, Available at http://www.math.uiuc.edu/Macaulay2/.
        

\bibitem [Mo]{Mo} D.Morrison {\it Isogenies between algebraic surfaces with geometric genus 1},Tokyo Journal of Mathematics 10.1 (1987), 179-187.

\bibitem [Mir] {Mir} R.Miranda, {\it Triple covers in algebraic Geometry}, American J. of Math, vol 107 (1985), 1123-1158.

\bibitem[NS]{NS} J.Nagel and M.Saito,{\it Relative Chow-K\"unneth decompositions for conic bundles and Prym varieties}, Int.Math.Res.Not.IMRN {\bf 16},(2009),2978-3001.
 
 \bibitem [Ped] {Ped}C.Pedrini, {\it On the finite dimensionality of a K3 surface}, Manuscripta Math. 138,59-72 (2012)


\bibitem [Shen]{Shen} M.Shen,{\it Surfaces with involution and Prym constructions}, arXiv:1209.5457 

\bibitem[SYZ]{SYZ} J. Shen, Q,Yin and X.Zhao, {\it derived categories of K3 surfaces, O'Grady's filtration and zero cycles on holomorphic symplectic varieties},arXiv:1705.06953v1[math.AG]


\bibitem [SV 1]{SV 1} M.Shen and C.Vial, {\it On the Chow groups of the variety of lines of a cubic fourfold}, arXiv:1212.0552v1 [math.AG] ,Dec. 2012.

\bibitem [SV 2]{SV 2} M.Shen and C.Vial, {\it The Fourier transform for certain Hyperk\"alher fourfolds}, Memoirs of the AMS 240, no. 1139, (2014), 1-104.
 
\bibitem [Shi]{Shi}T.Shirane ,{\it A note on normal triple covers over $\P^2$ with branch divisors of degree 6}, Branched coverings, degenerations and related topics 2013,
Tokyo Metropolitan University (Japan) (2013 3.7).
\bibitem[Tab]{Tab} G.Tabuada,{\it Schur-finiteness (and Bass-finitess)conjecture for quadric fibrations and families of sextic Du Val del Pezzo  surfaces}, arXiv:1708.05382v6[ Math.AG] 13 Mar.2019

\bibitem [TZ]{TZ} Z.Tian and R.Zong, {\it On cycles on rationally connected varieties}, Compos. Math. 150, no. 3, (2014),396-408.


 \bibitem [Vial 1] {Vial 1}C.Vial. {\it Algebraic cycles and fibrations}, Doc.Math {\bf 18},(2013), 1521-1553.

\bibitem [Vial 2] {Vial 2} C.Vial {\it Pure motives with representable Chow groups},   Comptes Rendus 348 (2010) 1191Ð1195

\bibitem[Vois 1]{Vois 1}C.Voisin, {\it Theoreme de Torelli pour le cubiques de $\P^5$}, Inv.Matn. {\bf b6}, (1986) 577-601.

  \bibitem [Vois 2]{Vois 2} C.Voisin.{\it Intrinsic pseudo-volume forms and K-correspondences},The Fano Conference, Univ.Torino,Turin (2004), 761-792. 
  
  \bibitem [Vois 3] {Vois 3} C.Voisin.{\it  Bloch's conjecture for Catanese and Barlow surfaces}, J. Differential Geometry 97 (2014) 149-175.
 
\bibitem [Yang]{Yang} J. Yang, {\it On quintic surfaces of general type}, Trans. Am.Math.Soc.  Vol. 295  no.2  (1986) 431-473.
 
 \bibitem[Za]{Za}Y.Zarhin . {\it Algebraic cycles over cubic fourfolds}, Boll.Un.Mat Ital. Vol 7,No4-B(1990),833-847.
 
  
 \end{thebibliography}
  \end{document}